\documentclass[11pt]{article}

\usepackage{amsmath,amssymb,amsfonts,amsthm,subfloat,algorithm,url,float,epsf,psfrag,bm}
\usepackage{mathtools}
\newcommand\blfootnote[1]{
    \begingroup
    \renewcommand\thefootnote{}\footnote{#1}
    \addtocounter{footnote}{-1}
    \endgroup
}

\usepackage[pdftex]{color,graphicx}
\usepackage[colorlinks,citecolor=blue,urlcolor=blue]{hyperref}

\usepackage[noabbrev,capitalize]{cleveref}
\usepackage{authblk}

\usepackage{hyperref}
\usepackage{algorithm, algorithmic, verbatim, float}
\usepackage{marvosym}
\usepackage{titlesec}
\titlespacing*{\section}{0pt}{0.6ex plus 0.6ex minus .2ex}{0.6ex plus .2ex}
\titlespacing*{\subsection}{0pt}{0.6ex plus 0.6ex minus .2ex}{0.6ex plus .2ex}
\usepackage[pagewise]{lineno}
\usepackage{setspace,caption}
\usepackage{verbatim}
\usepackage{array, float}
\usepackage{xcolor}
\usepackage{fontenc}
\usepackage[toc,page]{appendix}
\usepackage[nottoc]{tocbibind}
\usepackage[english]{babel}
\usepackage{relsize}
\usepackage{cellspace}
\usepackage{booktabs}
\usepackage{subcaption}
\usepackage[utf8]{inputenc}

\usepackage{mathrsfs}
\usepackage{amsfonts}
\usepackage{enumerate}
\usepackage{latexsym}
\usepackage{multirow}
\allowdisplaybreaks
\newtheoremstyle{exampstyle}
{8pt} % Space above
{8pt} % Space below
{\it} % Body font
{} % Indent amount
{\bfseries} % Theorem head font
{.} % Punctuation after theorem head
{.5em} % Space after theorem head
{} % Theorem head spec (can be left empty, meaning `normal')

\theoremstyle{exampstyle}
\usepackage{bbm}
\newtheorem{theorem}{Theorem}
\newtheorem{example}{Example}
\newtheorem{lemma}{Lemma}

\newtheorem{remark}{Remark}

\newtheorem{prop}{Proposition}

\newtheorem{defn}{Definition}

\numberwithin{equation}{section}
\numberwithin{example}{section}
\numberwithin{theorem}{section}
\numberwithin{lemma}{section}
\numberwithin{corollary}{section}
\numberwithin{prop}{section}
\numberwithin{defn}{section}
\numberwithin{remark}{section}

\newtheorem{assumption}[theorem]{Assumption}

\newcommand{\X}{\mathcal{X}}
\newcommand{\Y}{\mathcal{Y}}
\newcommand{\Ps}{\mathcal{P}}
\newcommand{\veps}{\varepsilon}

\newcommand{\A}{\mathcal{A}}

\newcommand{\eat}[1]{}

\newcommand{\ThetaBar}{\overline{\theta}}

\DeclareMathOperator*{\argmin}{\arg\!\min}
\DeclareMathOperator*{\argmax}{\arg\!\max}

\newcommand{\nc}{\normalcolor}
\newcommand{\N}{\mathbb{N}}

\newcommand{\keywords}[1]{\textbf{\textit{Keywords:}} #1}
\renewcommand{\bar}[1]{\overline{#1}}
\renewcommand{\hat}[1]{\widehat{#1}}
\renewcommand{\tilde}[1]{\widetilde{#1}}
\newcommand{\E}{\mathbb{E}}

\newcommand{\R}{\mathbb{R}}
\newcommand{\cc}{\mathbf{c}}

\newcommand{\eps}{\epsilon}

\definecolor{LightCyan}{rgb}{0.88,1,1}
\definecolor{Gray}{gray}{0.9}

\usepackage{xr}
\usepackage{enumerate}
\begin{document}

\title{A New Perspective On Denoising Based On Optimal Transport}

% \author{Nicol\'{a}s Garc\'{i}a Trillos \\ email \href{mailto:me@somewhere.com}{me@somewhere.com} \\ Department of Statistics\\ University of Wisconsin-Madison 
%   \and Bodhisattva Sen \\ email \href{mailto:someone@somewhere.com}{someone@somewhere.com} \\ Department of Statistics\\ Columbia University }

\author[1]{Nicol\'{a}s Garc\'{i}a Trillos}
\affil[1]{Department of Statistics\\ University of Wisconsin-Madison}
\affil[1]{email: garciatrillo@wisc.edu}
\vspace{0.1in}

\author[2]{Bodhisattva Sen}
\affil[2]{Department of Statistics\\ Columbia University}
\affil[2]{email: bodhi@stat.columbia.edu }
\date{\today}

%\begin{frontmatter}

%\title{A Mean Field Approach to Empirical Bayes Estimation in High-dimensional Linear Regression}
%\title{A sample article title with some additional note\thanksref{T1}}
%\runtitle{Empirical Bayes High-dimensional Linear Regression}
%\thankstext{T1}{A sample of additional note to the title.}

%\begin{aug}
%%%%%%%%%%%%%%%%%%%%%%%%%%%%%%%%%%%%%%%%%%%%%%%
%% Only one address is permitted per author. %%
%% Only division, organization and e-mail is %%
%% included in the address.                  %%
%% Additional information can be included in %%
%\author[A]{\fnms{Sumit}~\snm{Mukherjee\thanks{Supported by NSF DMS-2113414.}}\ead[label=e1]{sm3949@columbia.edu}}\and
%\author[B]{\fnms{Bodhisattva}~\snm{Sen\thanks{Supported by NSF DMS-2311062.}}\ead[label=e2]{bodhi@stat.columbia.edu}} \and
%\author[C]{\fnms{Subhabrata}~\snm{Sen\thanks{Supported by NSF DMS-CAREER 2239234, ONR N00014-23-1-2489 and AFOSR FA9950-23-1-0429.}}\ead[label=e3]{subhabratasen@fas.harvard.edu}}
%%%%%%%%%%%%%%%%%%%%%%%%%%%%%%%%%%%%%%%%%%%%%%
%% Addresses                                %%
%%%%%%%%%%%%%%%%%%%%%%%%%%%%%%%%%%%%%%%%%%%%%%
%\address[A]{Department of Statistics, Columbia University\printead[presep={,\ }]{e1}}
%\address[B]{Department of Statistics, Columbia University\printead[presep={,\ }]{e2}}
%\address[C]{Department of Statistics, Harvard University\printead[presep={,\ }]{e3}}

%\end{aug}
 
%\begin{document}

\maketitle

\abstract{In the standard formulation of the classical denoising problem, one is given a probabilistic model relating a latent variable $\Theta \in \Omega \subset \R^m \; (m\ge 1)$ and an observation $Z \in \R^d$ according to:  $Z \mid \Theta \sim p(\cdot\mid \Theta)$ and $\Theta \sim G^*$, and the goal is to construct a map to recover the latent variable from the observation. The posterior mean, a natural candidate for estimating $\Theta$ from $Z$, attains the minimum Bayes risk (under the squared error loss) but at the expense of over-shrinking the $Z$, and in general may fail to capture the geometric features of the prior distribution $G^*$ (e.g., low dimensionality, discreteness, sparsity). To rectify these drawbacks, in this paper we take a new perspective on this denoising problem that is inspired by optimal transport (OT) theory and use it to study a different, OT-based, denoiser at the population level setting. We rigorously prove that, under general assumptions on the model, this OT-based denoiser is mathematically well-defined and unique, and is closely connected to the solution to a Monge OT problem. We then prove that, under appropriate identifiability assumptions on the model, the OT-based denoiser can be recovered solely from information of the marginal distribution of $Z$ and the posterior mean of the model, after solving a linear relaxation problem over a suitable space of couplings that is reminiscent of standard multimarginal OT (MOT) problems. In particular, thanks to Tweedie's formula, when the  likelihood model $\{ p(\cdot \mid \theta) \}_{\theta \in \Omega}$ is an exponential family of distributions, the OT based-denoiser can be recovered solely from the marginal distribution of $Z$. In general, our family of OT-like relaxations is of interest in its own right and for the denoising problem suggests alternative numerical methods inspired by the rich literature on computational OT.}

\medskip


% \boxedtext{
% \begin{itemize}
% \item Key boxed text here.
% \item Key boxed text here.
% \item Key boxed text here.
% \end{itemize}}

\maketitle

\section{Introduction}

\label{sec:Intro}

\blfootnote{This work will appear in \textit{Information and Inference: A Journal of the IMA}.}

Consider the following simple latent variable model: 
\begin{equation}\label{eq:Mix-Mdl}
Z \mid\Theta = \theta \stackrel{}{\sim} p(\cdot\mid \theta) \quad \mbox{and} \quad  \Theta  \stackrel{}{\sim} G^*, 
\end{equation}
where $\{ p(\cdot \mid \theta) \}_{\theta \in \Omega}$ is a known parametric family of probability density functions (p.d.f.'s) on $\R^d$ ($d\ge 1$) with respect to (w.r.t.) the Lebesgue measure, and $G^*$ is a probability distribution whose support is contained in the set $\Omega$, a subset of $\R^m$ for $m\ge 1$. We only get to observe $Z$ from the above model and $\Theta$ is the unobserved latent variable of interest. We denote by $P_{Z,\Theta}$ the joint distribution of $(Z,\Theta)$ on ${\R^d} \times \Omega$. By defining a joint distribution over the observable $Z$ and the latent variable $\Theta$, the corresponding distribution of the observed variable is then obtained
by marginalization; $Z$ has marginal distribution $\mu$ with density (w.r.t.~the Lebesgue measure)
\begin{equation}\label{eq:Marg-Dens}
f_{G^*}(z) := \int p(z \mid \theta) \, dG^*(\theta), \qquad \mbox{for } z \in {\R^d}.
\end{equation}
Such latent variable models allow relatively complex marginal distributions to be expressed in terms of more tractable joint distributions over the expanded variable space and thus they provide an important tool for the analysis of multivariate data. Note that~\eqref{eq:Mix-Mdl} captures a conceptual framework within which many disparate methods can be unified, including mixture models, factor models, etc; see e.g.,~\cite{Bartholomew-Et-Al-2011}. In fact,~\eqref{eq:Mix-Mdl} can be thought of as a simple Bayesian model where the prior distribution on $\Theta$ is $G^*$. A few important examples of such a setting are given below.% (other prominent examples include uniform scale mixtures, Poisson mixtures, etc.; see~\cref{Appendix:Examples} for the details).
\begin{example}[Normal location mixture]\label{ex:Norm-loc-Mix}
	Suppose that $p(z \mid \theta) = \varphi_\sigma(z - \theta)$ where $\varphi_\sigma(\cdot)$ is the p.d.f.~of the multivariate normal distribution with mean $0$ and variance $\sigma^2I_d$ ($\sigma^2$ known), i.e., $\varphi_\sigma(z) := \frac{1}{(\sqrt{2 \pi \sigma^2})^d} \exp(-\frac{z^\top z}{2 \sigma^2})$, for $z \in \R^d$; here $m = d$. If $G^*$ is a discrete distribution with finitely many atoms, then $Z$ comes from a finite Gaussian mixture model. This model is ubiquitous in statistics and arises in many application domains including clustering; see e.g.,~\cite{EM-Algo-1977, McLachlan-Peel-2000}. 
\end{example}

\begin{example}[Normal scale mixture]\label{ex:Norm-scale-Mix}
	Suppose that $p(z \mid \theta) = \frac{1}{\theta}\varphi(\frac{z}{\theta})$ where $\varphi(\cdot)$ is the p.d.f.~of the standard normal distribution on $\R$. Here $G^*$ is a probability distribution on the positive real line $(0,\infty)$. This corresponds to the Gaussian scale mixture model; see~\cite{Andrews-Mallows-74}. This model has many applications including in Bayesian (linear) regression and multiple hypothesis testing, see e.g.,~\cite{West-84, Park-Casella-2008, Stephens-2017}.
\end{example}

\begin{example}[Uniform scale mixture]\label{ex:Unif-scale-Mix}
	Suppose that $G^*$ is a distribution on $(0,\infty)$ and $p(\cdot \mid \theta)$ corresponds to the uniform density on the interval $[0,\theta]$ (for $\theta >0$). Thus, the marginal density of $Z$ is given by $f_{G^*}(z) := \int \frac{1}{\theta} \mathbb{I}_{[0,\theta]}(z) \, dG^*(\theta) = \int_z^\infty \frac{1}{\theta} \, dG^*(\theta),$ for $z >0$. It is well-known that any (upper semicontinuous) nonincreasing density on $(0,\infty)$ can be represented as $f_{G^*}$ for a suitable $G^*$ \cite[p.~158]{Feller-71}. This class of distributions
	arises naturally via connections with renewal theory (see e.g.,~\cite{Woodroofe-Sun-1993}), multiple testing (see e.g.,~\cite{Langaas-Et-Al-2005},~\cite{Ignatiadis-Huber-21}), etc.
\end{example}

%{\color{blue} We can also include examples where $Z|\Theta =\theta$ has a Bernoulli/binomial distribution with success probability $\theta \in (0,1)$, a categorical/multinomial  distribution with probability vector $\theta$, etc. }

We consider the goal of estimating the unobserved $\Theta$ in~\eqref{eq:Mix-Mdl}; we call this task that of {\it denoising} $Z$. Traditionally, this goal has been formulated as that of finding an estimator $\mathfrak{d}^*(\cdot)$ that minimizes the {\it Bayes risk} w.r.t.~a {\it loss function} $\ell:\R^m \times \R^m \to [0,\infty)$, i.e.,
\begin{equation}\label{eq:Bayes-risk}
\E \left[\ell(\mathfrak{d}(Z),\Theta) \right] \equiv \int_\Omega \int_{\R^d} \ell(\mathfrak{d}(z),\theta) \, p(z \mid  \theta) \, dz \, dG^*(\theta)
\end{equation}
over all measurable functions $\mathfrak{d}: \R^d \to \R^m$, where $(Z,\Theta) \sim P_{Z,\Theta}$ (i.e., $\Theta \sim G^*$ and $Z \mid \Theta = \theta \sim p(\cdot \mid \theta)$). The best estimator $\mathfrak{d}^*(Z)$ of $\Theta$, in terms of minimizing~\eqref{eq:Bayes-risk}, is called the {\it Bayes estimator} under the  loss $\ell(\cdot,\cdot)$. 
\begin{example}[Bayes estimator under squared error loss] When we use the loss function $\ell(a,\theta) := |a - \theta|^2$ (here  $a,\theta \in \R^m$ and $|\cdot|$ denotes the usual Euclidean norm), the Bayes estimator $\overline{\theta}(\cdot)$ minimizing~\eqref{eq:Bayes-risk} turns out to be the {\it posterior mean}, i.e.,
	%\begin{equation}\label{eq:Bayes-risk-L_2}
	%\E \left[\|\mathfrak{d}(Z) - \Theta\|^2 \right] \equiv \int \| \mathfrak{d}(Z) - \theta \|^2 \, p(z| \theta) d \lambda \, dG^*(\theta)
	%\end{equation}
	\begin{equation}\label{eq:Bayes-Est}
	\overline{\theta}(Z) := \E[\Theta \mid Z].
	\end{equation}
\end{example}

In this paper we take a different perspective on the denoising problem inspired by the theory of optimal transport (OT). To motivate our approach to estimating the unobserved $\Theta$ in~\eqref{eq:Mix-Mdl}, we first highlight a drawback of the Bayes estimator. Although the {\it posterior mean} $\bar \theta(Z) \equiv \E[\Theta \mid Z]$ in~\eqref{eq:Bayes-Est} attains the smallest Bayes risk  (see~\eqref{eq:Bayes-risk}) among all estimators of $\Theta$ (under the squared error loss), its distribution is different from $G^*$ (recall that $\Theta \sim G^*$). In fact, in some cases the Bayes estimator $\bar \theta(Z)$ 
yields a `shrunken' estimate of $\Theta$. The left panel of Figure~\ref{fig:Motivation} illustrates this with $n=60$ data points $Z_1,\ldots, Z_n$ (denoted by the {\color{blue} blue} dots) drawn from the model $Z_i \mid \Theta_i = \theta \sim N(\theta,1)$ where $\Theta_i \stackrel{iid}{\sim} G^*$ with $G^* = N(0,\tau^2)$ and $\tau^2=1$. The latent $\Theta_i$'s are denoted by {\color{red} red} dots, whereas the Bayes estimator $\bar \theta(Z_i)$ is depicted by black dots. We can see that the Bayes estimator (excessively) shrinks the observations in order to achieve optimal denoising (compare the distributions of the {\color{red} red} and the black dots). The resulting distribution of the Bayes estimators $\bar \theta(Z)$ is $N(0,\frac{1}{2})$, which has a much smaller variance than $G^* \equiv N(0,{1})$. 

\begin{figure}
	\centering
	\includegraphics[scale=0.48]{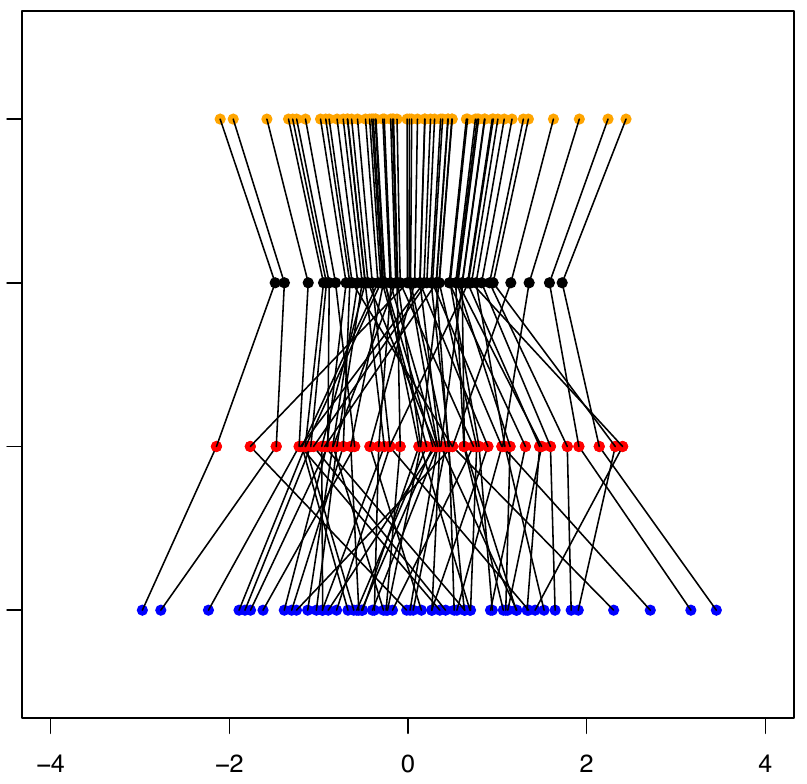}
	\includegraphics[scale=0.48]{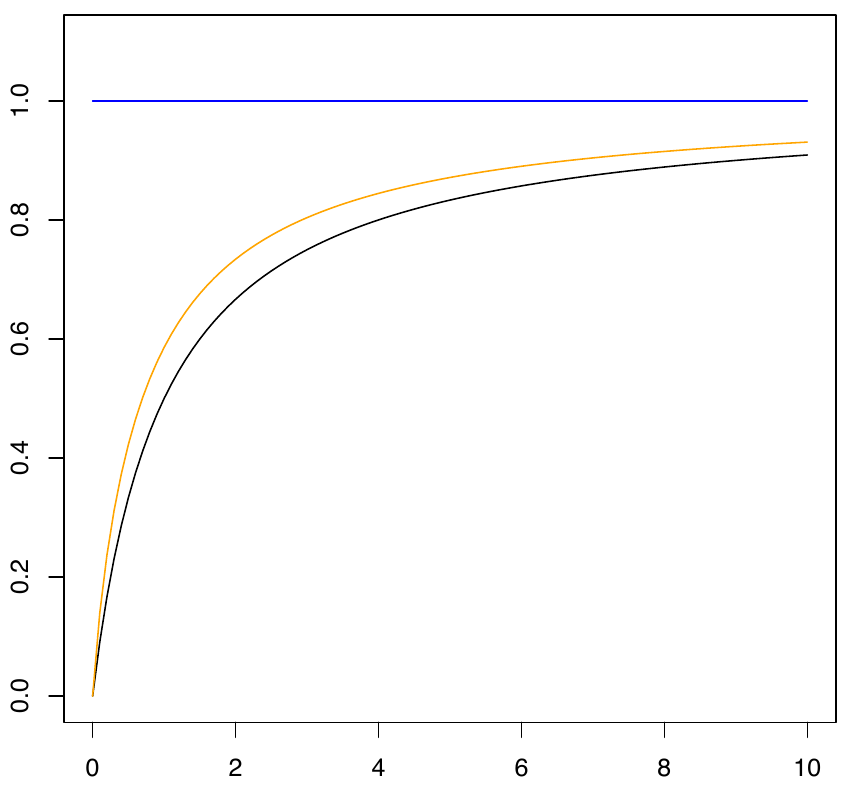} 
	\caption{Toy example with $n = 60$ in $d=1$ where $p(\cdot \mid \theta)$ is the density of  $N(\theta, 1)$ and $G^* = N(0, \tau^2)$. \textbf{Left:} Observations $Z_1,\ldots, Z_n$ (in {\color{blue} blue}) obtained from model~\eqref{eq:Mix-Mdl} with $\tau^2 = 1$ are connected to their true unobserved latent variables $\{\Theta_i\}_{i=1}^n$ (in {\color{red} red}); the Bayes estimator $\bar \theta(Z_i)$ (in {\color{black} black}) is connected to $\Theta_i$ (in {\color{red} red}) and the corresponding OT-based denoiser $\delta^*(Z_i)$ (in {\color{orange} orange}). \textbf{Right:} Plot of the risk curves of the three estimators of $\Theta$ --- $Z$ (in {\color{blue} blue}),  $\bar \theta(Z)$  (in black) and $\delta^*(Z)$ (in {\color{orange} orange}) --- as $\tau^2$ varies from 0 to 10.  }\label{fig:Motivation}
\end{figure}
In contrast, in this paper we consider the {\it OT-based denoiser} $\delta^*(Z)$ (see~\eqref{eq:delta-OT}), %Definition~\ref{def:BarycentricOracleEstimator})
shown in the left plot of Figure~\ref{fig:Motivation} by the {\color{orange} orange} dots, which corrects this drawback and produces estimates that have the distribution $G^*$; compare the distributions of the {\color{orange} orange} and the {\color{red} red} dots. The plot of the risk functions for the three estimators --- $Z$,  $\bar \theta(Z)$ and $\delta^*(Z)$ --- as $\tau^2$ varies show that the proposed {OT-based denoiser} $\delta^*(Z)$ achieves the distributional stability (i.e., $\delta^*(Z) \sim G^*$) at very little cost; compare the risk functions for $\delta^*(Z)$ (in {\color{orange} orange}) and $\bar \theta(Z)$ (in black). See Remark~\ref{rem:Nor-Nor} for the detailed computations.

%In general, reasonable loss function in the setting of~\eqref{eq:Mix-Mdl} would be $\ell(\mathfrak{d}(Z),\Theta) = -\log p(Z| \mathfrak{d}(Z))$ generalizes the quadratic loss $\|\cdot - \cdot\|^2$ to tackle (b)\footnote{{\color{red} This could also take care of invariance issues when estimating a function of $\Theta$, instead of $\Theta$}.}.
%\begin{figure}
%\centering
%\includegraphics[scale=0.48]{Theta-Z.pdf}
%\includegraphics[scale=0.48]{Z-Bayes-hat-Z.pdf} 
%\caption{Toy data with $n = 60$ in $d=1$ where $p(\cdot \mid \theta)$ is the density of  $N(\theta, 1)$ and $G^* = N(0, 1)$. Left panel: Observations $Z_1,\ldots, Z_n$ (in blue) obtained from model~\eqref{eq:Mix-Mdl} are connected to their true unobserved latent variables $\{\Theta_i\}_{i=1}^n$ (in red). Right: Oracle Bayes estimators (in black) connected to their corresponding true (unobserved) $\Theta_i$'s (in red) and their oracle barycentric estimators (in orange).   }\label{fig:Motivation}
%\end{figure}

\begin{figure}
	\centering
	\includegraphics[scale=0.38]{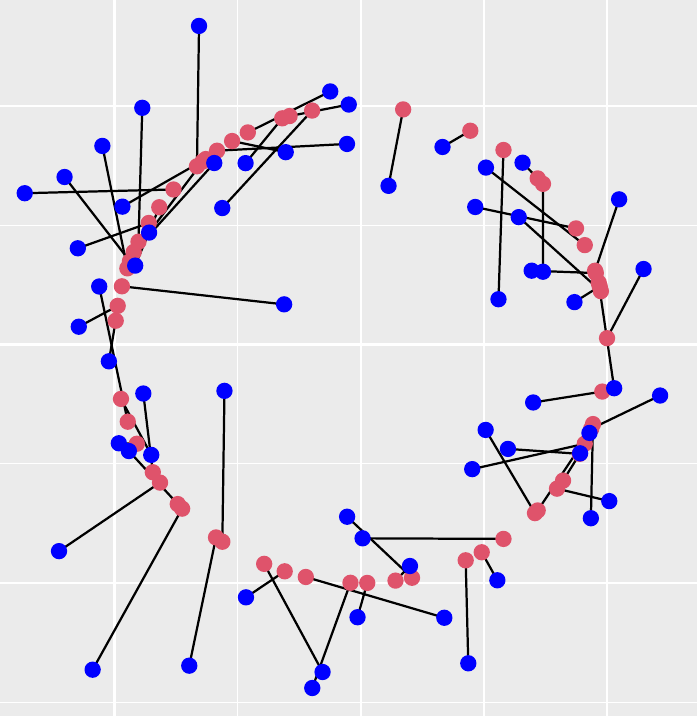} 
	\includegraphics[scale=0.485]{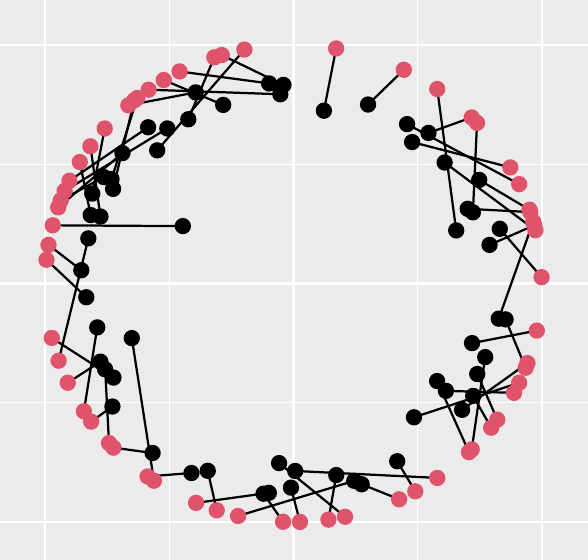} 
	\includegraphics[scale=0.49]{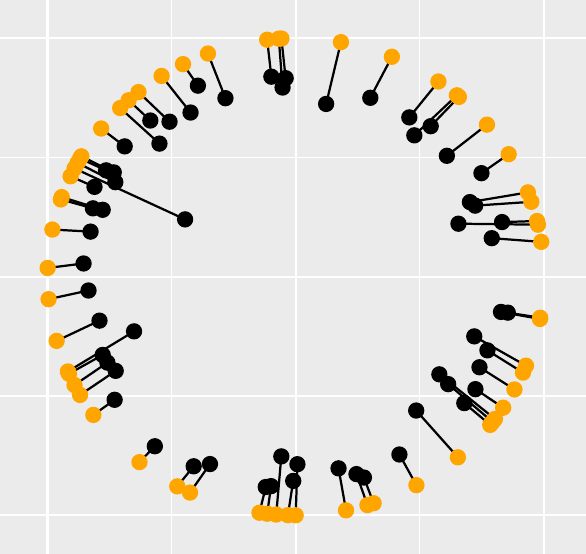} 
	\caption{Toy example with $n = 60$ in $d=2$ where $p(\cdot \mid \theta)$ is the density of  $N(\theta, (0.3)^2\cdot I_2)$ and $G^*$ is the uniform distribution on the unit circle. \textbf{Left:} Observations $Z_1,\ldots, Z_n$ (in {\color{blue} blue}) obtained from model~\eqref{eq:Mix-Mdl} are connected to the corresponding unobserved latent variables $\{\Theta_i\}_{i=1}^n$ (in {\color{red} red}). \textbf{Center:} The Bayes estimator $\bar \theta(Z_i)$ (in black) is connected to $\Theta_i$ (in {\color{red} red}), for every $i=1,\ldots, n$.  \textbf{Right:} The Bayes estimator $\bar \theta(Z_i)$ (in black) is connected to its corresponding OT-based denoiser $\delta^*(Z_i)$ (in {\color{orange} orange}) lying on the circle.    }\label{fig:Normal-Normal}
\end{figure}

This (over)-shrinkage by the Bayes estimator $\bar \theta(Z)$ is more acute when $d \ge 2$. In general, the Bayes estimator $\bar \theta(Z)$ is not necessarily guaranteed to lie `close' to $\text{spt}(G^*)$, the support\footnote{The smallest closed set containing probability mass 1.} of $G^*$ (recall that $\Theta \sim G^*$). To illustrate this,  in Figure~\ref{fig:Normal-Normal} we consider another example, this time with $d=m=2$. Here we take $n=60$ data points $Z_1,\ldots, Z_n \in \R^2$ (depicted by the {\color{blue} blue} dots in the left panel of Figure~\ref{fig:Normal-Normal}) drawn from the normal location mixture model (Example~\ref{ex:Norm-loc-Mix}) with the latent variables $\Theta_1, \ldots, \Theta_n$ drawn uniformly on the circle of radius 1 (shown by the {\color{red} red} dots in the left panel of Figure~\ref{fig:Normal-Normal}). We connect $Z_i$ with $\Theta_i$ by a black line, for each $i=1,\ldots, n$, in the plot. The middle panel of Figure~\ref{fig:Normal-Normal} shows the Bayes estimator\footnote{Here, the Bayes estimator, defined in~\eqref{eq:Bayes-Est}, is approximated by a fine discretization of $G^*$, i.e., $G^* \approx \frac{1}{M} \sum_{i=1}^M \delta_{a_i}$ where $M = 1200$ and the $a_i$'s lie uniformly on the circle.} at the observed data points $\bar \theta(Z_i)$ (depicted by the black dots) connected to the corresponding $\Theta_i$'s (in {\color{red} red}). As can be easily seen from this plot, the Bayes estimator shrinks most of the observations towards $0 \in \R^2$. In contrast, our proposed {OT-based denoiser} corrects this drawback and maps the Bayes estimator $\bar \theta(Z_i)$ to $\delta^*(Z_i)$ (shown in the right panel of Figure~\ref{fig:Normal-Normal} by {\color{orange} orange} dots) which lies on the circle. Note that $\delta^*(Z_i)$, by definition, takes values in $\text{spt}(G^*)$, the support of $G^*$. %Note that each  the unknown $\Theta_i$'s (in red) by the corresponding blue points. 

%The right panel of Figure~\ref{fig:Normal-Normal} shows the OT-based denoiser $\delta^*(Z_i)$ (in {\color{orange} orange} dots) connected to the corresponding Bayes estimator $\bar \theta(Z_i)$. 

In fact, if the goal is to estimate $\Theta \sim G^*$, it is reasonable to restrict $\mathfrak{d}(\cdot)$ in~\eqref{eq:Bayes-risk} to all estimators such that $\mathfrak{d}(Z)$ is distributed (approximately) as $G^*$. This type of requirement has been explored in previous works in the literature (see e.g.,~\cite{Jasa1984,NEURIPS2021_d77e6859}) and is particularly important when we believe that $G^*$ is discrete with a few atoms (which corresponds to the {\it clustering problem}) or when we believe that $G^*$ has `structure' (e.g., supported on a lower dimensional manifold in $\R^m$). 
%	\item[(b)] The minimization of the Bayes risk in~\eqref{eq:Bayes-risk} seems intimately tied to the squared error loss. Probably a more suitable model dependent loss function to minimize would be to consider the structure of the known parametric model $p(\cdot|\cdot)$. 
%\end{enumerate}
In light of this discussion, it is natural to seek solutions to  
\begin{equation}\label{eq:Oracle-Bayes-risk}
\inf_{\delta: \R^d \rightarrow \R^m } \E_{(Z, \Theta)\sim P_{Z,\Theta} }\left[|\delta(Z)- \Theta|^2 \right]  \qquad \mbox{subject to }\;\; \delta(Z) \sim G^*,
\end{equation}
among all measurable functions $\delta: {\R^d} \to \R^m$, where we consider $\ell(\cdot,\cdot)$ to be the squared error loss for simplicity; see Appendix~\ref{sec:General-cost} for a discussion on more general loss functions.  The constraint $\delta(Z) \sim G^*$ ensures that the `estimator' $\delta(Z)$ of $\Theta$ has the same distribution as $G^*$ (in particular, the same support as $\Theta \sim G^*$), thereby addressing the above drawbacks; cf.~\eqref{eq:Bayes-risk}. Solutions of \eqref{eq:Oracle-Bayes-risk}, if they exist, can be described as \textit{distortion} minimizers under a perfect \textit{perception} quality constraint; see \cite{NEURIPS2021_d77e6859,BlauMichaeli} for definitions and motivation for this terminology. \nc

\nc 
%{\color{red} 
% Problem \eqref{eq:Oracle-Bayes-risk} can be interpreted as a \textit{G-model} for the denoising problem, given that the constraint $\delta_{\sharp } \mu = G^* $ explicitly involves the prior distribution $G^*$.} 
The above discussion leads to the following natural questions: 
\begin{enumerate}
	\item[\textbf{Q1.}] Under what conditions can it be guaranteed that there exist solutions to problem \eqref{eq:Oracle-Bayes-risk}? Are these solutions unique?
\item[\textbf{Q2.}] If there exists a solution, how can one characterize it, and what potential approaches can one follow to find it?
 \item[\textbf{Q3.}]  Is it possible to obtain a solution solely based on the marginal distribution of observations and knowledge of the likelihood model (without explicitly using $G^*$)? \nc 
\end{enumerate}

The purpose of this paper is to provide answers to the above questions in the population level setting, implicitly also assuming that $G^*$ is known. In this process, we lay down some mathematical foundations and outline some strategies for future implementation of our ideas in finite data settings, with known or unknown $G^*$. Q3 is motivated by the fact that, in general, an attempt to recover $G^*$ from observations can lead to a difficult deconvolution problem; see more discussion below.

Our first main result, \textbf{Theorem}~\ref{thm:Main1}, states that, under certain assumptions on the model, problem~\eqref{eq:Oracle-Bayes-risk} indeed possesses a unique solution $\delta^*(\cdot)$; throughout the paper, by unique solution we mean unique $\mu$-a.e. What is more, this solution can be found by solving an OT problem (defined precisely in \eqref{eq:OT_PushForward}) between the distribution of the Bayes estimator $\overline{\theta}(Z)$ and $G^*$, and can be characterized as the composition of the gradient of a certain convex function and $\overline{\theta}(\cdot)$\nc. We refer to this $\delta^*(\cdot)$ as the \textit{OT-based denoiser} associated to the model $(\{ p(\cdot\mid \theta) \}_{\theta \in \Omega}, G^*)$. Problem \eqref{eq:Oracle-Bayes-risk} can also be interpreted as an extreme case in a family of problems with a soft penalty defined according to
\begin{equation}
\label{eq:EquivProblemWithPnelizatin}
\inf_{\delta: {\R^d} \rightarrow \R^m } \E_{(Z, \Theta)\sim P_{Z,\Theta} } \left[ | \delta(Z) - \Theta|^2  \right] + \frac{1}{2 \tau} W_2^2(\delta_{\sharp} \mu, G^*) , 
\end{equation}
where $\tau >0$ is a tuning parameter, $\delta_{\sharp} \mu $ is the pushforward of $\mu$ by $\delta$ (i.e., the distribution of $\delta(Z)$ if $Z \sim \mu$; see Definition~\ref{defn:Pushforward}) and $W_2(\cdot, \cdot)$ denotes the $2$-Wasserstein distance between probability distributions (see Definition~\ref{def:2Wass}). Formally, when $\tau\rightarrow \infty$ we recover the standard unconstrained risk minimization problem, whose solution is the Bayes estimator, whereas we recover problem \eqref{eq:Oracle-Bayes-risk} when $\tau \rightarrow 0$. For any other value of $\tau$ in between these two extremes, \textbf{Theorem}~\ref{thm:Main2} guarantees that the solution to \eqref{eq:EquivProblemWithPnelizatin} is unique and can be explicitly written as a simple linear interpolation of the OT-based denoiser $\delta^*(\cdot)$ and the Bayes estimator $\overline{\theta}(\cdot)$, a result very closely related to the characterization of the so called \textit{distortion-perception tradeoff} in
Wasserstein space established in \cite{NEURIPS2021_d77e6859}. \nc We will refer to \eqref{eq:EquivProblemWithPnelizatin} as a \textit{latent space penalization} approach to denoising, given that the penalty term $W_2^2(\delta_{\sharp}\mu, G^*)$ involves an explicit comparison of distributions in the latent space $\Omega \subset \R^m$. {It is worth highlighting that other optimization problems similar to \eqref{eq:EquivProblemWithPnelizatin} have been considered in papers such as \cite{BlauMichaeli,WangWen} (see also references therein), where the $W_2$ distance between measures is substituted by other metrics over probability measures, including the $1$-OT distance $W_1$ and other loss functions as used in generative adversarial networks (GANs). As mentioned earlier, and in contrast to the aforementioned papers, in this paper we pursue an in depth analysis of the properties of solutions to problems like \eqref{eq:EquivProblemWithPnelizatin} (or \eqref{eq:FModel} below) and suggest novel strategies to find them.

% {\color{red} In this paper we have assumed that the prior $G^*$ is known or can be adequately approximated so that we can solve the optimization problems~\eqref{eq:Oracle-Bayes-risk},~\eqref{eq:EquivProblemWithPnelizatin}, and~\eqref{eq:FModel}. 

% the fact that it explicitly depends on the distribution $G^*$ (or an estimator thereof) 
%  

%  Although we do not pursue this in our paper, our formulations ~\eqref{eq:EquivProblemWithPnelizatin} and~\eqref{eq:FModel} are reminiscent
% %  of $G$- and $F$-modelling commonly used in the empirical Bayes literature; see~\cite{Robbins-1956},~\cite{Efron11}.}

Although the characterization of the OT-based denoiser $\delta^*(\cdot)$ as a solution to an OT problem is appealing, in many real applications $G^*$ may be unknown, making this characterization difficult to implement. One possible approach to go around this issue is to estimate $G^*$ using i.i.d.~data from~\eqref{eq:Mix-Mdl} using tools from what is usually referred to in statistics as {deconvolution} (see e.g.,~\cite{carroll1988optimal, zhang1990fourier, fan1991optimal, meister2009deconvolution}). This approach is also taken in the {\it empirical Bayes} literature; see, e.g.,~\cite{Robbins-1956, Jiang-Zhang-2009, Efron-2011, Efron-19, Soloff-Et-Al-2021}, as well as the brief discussion on this topic that we present in Appendix \ref{sec:EmpiricalBayes}. In this paper, however, we offer an alternative approach and study yet another formulation for the denoising problem that closely resembles \eqref{eq:EquivProblemWithPnelizatin} but where we directly work with $\mu$, the (marginal) distribution of the observed data (see~\eqref{eq:Marg-Dens}). \nc Indeed, we consider the optimization problem:
\begin{equation}\label{eq:FModel}
\inf_{\delta: {\R^d} \to \R^m }   \mathcal{E}_{\tau}(\delta) :=  \E_{(Z,\Theta) \sim P_{Z,\Theta}} [  |\delta(Z) - \Theta|^2   ]  + \frac{1}{2\tau} W_2^2(\mu_\delta, \mu);
\end{equation}
here, for a given map $\delta$ we define $\mu_\delta$ as the probability measure over $\R^d$ defined as
\begin{equation}\label{eq:mu-delta}
\mu_\delta(A) := \int_A \int_{{\R^d}}  p(z'\mid \delta(z) )\,d \mu(z) \, dz' , \quad \forall A \subseteq \R^d \text{ Borel measurable}.  
\end{equation}
In words, $\mu_\delta$ is the marginal distribution of the variable $z$ assuming that the underlying distribution of the latent variable $\theta$ is given by $G= \delta_{\sharp} \mu$. We will refer to \eqref{eq:FModel} as an \textit{observable space penalization} approach to denoising, given that the penalty term $W_2^2(\mu_\delta, \mu)$ involves an explicit comparison of distributions in the observable space $\R^d$. In \textbf{Proposition}~\ref{thm:FrechetDiff} we show that, under suitable assumptions, the objective function $\mathcal{E}_\tau$ (in~\eqref{eq:FModel}) is Gateaux differentiable w.r.t.~the target $\delta \in \mathbb{L}^2({\R^d} : \R^m; \mu)$\footnote{$\mathbb{L}^2({\R^d} : \R^m; \mu)$ is the space of vector valued (equivalence classes of) measurable functions from ${\R^d}$ into $\R^m$ that are square-integrable w.r.t.~$\mu$.} and provide an explicit formula for its gradient (see~\eqref{eqn:GradientE}). The formula for the gradient, which can be easily adapted to the empirical setting, can in principle be used to implement a first order optimization method seeking a solution for~\eqref{eq:FModel}. Unfortunately, problem \eqref{eq:FModel} is non-convex in $\delta$ and one cannot guarantee the convergence of a steepest descent scheme towards a global minimizer of $\mathcal{E}_\tau(\cdot)$. %Furthermore, in the general setting when $\mu$ is a measure with an infinite support, the space $\mathbb{L}^2({\R^d} : \R^m; \mu)$ is infinite dimensional, and it is not even clear that the gradient flow of the energy $\mathcal{E}_\tau$ with respect to the $\mathbb{L}^2({\R^d}: \R^m;\mu)$ metric converges towards a local minimum if started arbitrarily. 
In fact, even the existence of global solutions to \eqref{eq:FModel} is not guaranteed by straightforward arguments in the calculus of variations. The main technical difficulty for this is the lack of lower semicontinuity of the functional $\delta \mapsto W_2^2(\mu_\delta, \mu)$ w.r.t.~the \textit{weak} topology in the Hilbert space $\mathbb{L}^2({\R^d} : \R^m ; \mu)$ (see Definition \ref{def:WeakL2}), a natural topology where one can guarantee pre-compactness of minimizing sequences. 

Despite the above discussion, we can prove that indeed there exist solutions to \eqref{eq:FModel}; see \textbf{Theorem}~\ref{cor:StructureSolutions}. This is achieved by considering a suitable relaxation argument where we ``lift" the original problem \eqref{eq:FModel} to a problem over couplings (see~\eqref{eqn:Relaxation} for details) that, while not of a standard type in OT theory, does resemble multimarginal optimal transport (MOT) problems. Like MOT problems, our relaxation is linear, and its search space enjoys better compactness properties than the original problem \eqref{eq:FModel} that in particular can be used to prove existence of solutions (see \textbf{Theorem}~\ref{lem:StructureGammastar}). This relaxation, which we show is exact under suitable assumptions, also motivates the use of computational tools in OT for constructing solutions of \eqref{eq:FModel}; this will be explored in future work. Finally, we highlight that this relaxation is the key mathematical construction that allows us to prove \textbf{Theorem}~\ref{thm:Soft-F}, which states that, under the identifiability assumptions on the probabilistic model that are written down precisely in Assumption~\ref{assump:Identifiability}, the solutions $\delta_\tau^*$ of \eqref{eq:FModel}
converge, as $\tau \rightarrow 0$, to the OT-based denoiser $\delta^*$; in Remark \ref{rem:NonIdentifiable} we discuss the non-identifiable case.

As we discuss in Section~\ref{sec:Discussion}, in order to use the relaxation problem \eqref{eqn:Relaxation} to approximate $\delta^*$ from finitely many observations, one would first need to estimate $\overline{\theta} (\cdot)$ from the available data. This is where Tweedie's formula (see~\eqref{eq:Tweedie} in Appendix~\ref{sec:Tweedie}) can be very useful. This formula expresses the posterior mean $\overline{\theta}(\cdot)$ in an exponential family model (see Appendix~\ref{sec:Exp-Family}) in terms of the marginal density $f_{G^*}$ of the observations (and its gradient) only, and can thus be estimated (nonparametrically) directly from observations $Z_1,\ldots, Z_n$, say via kernel density estimation. We thus anticipate to be able to construct consistent estimators for $\delta^*$ without knowing $G^*$ explicitly or having to directly estimate it, at least in the case when the likelihood model is an exponential family of distributions.

\subsection{Previous works}
The main motivation of our work comes from the theory of empirical Bayes (\cite{Robbins-1956}) and its recent revisitations (see e.g.,~\cite{Efron-2003, Efron-2009, Efron-2010, Efron-19}) which consider large data sets that arise from parallel and similar experiments. In the classical empirical Bayes setup the unknown parameters arising from the parallel experiments are assumed to be i.i.d.~random variables with an unknown common prior distribution $G^*$. 

Typically, empirical Bayes methodologies (see e.g.,~\cite{kiefer1956consistency, Laird-1978, Bohning-1999, Jiang-Zhang-2009, Efron-2011, Efron-2014, KM-14, Efron-2016, Gu-Koenker-2017, Koenker-Gu-2019, Soloff-Et-Al-2021, Zhong-et-al-2022, Gu-Koenker-2023} and the references therein) provide statistical procedures which approximate the Bayes rule for the true model (without specifying a prior). In this paper we question this very premise and illustrate (cf.~Figures~\ref{fig:Motivation} and~\ref{fig:Normal-Normal}) that the Bayes estimator, which is optimal in terms of squared error risk, deforms the underlying true distribution of the latent variables (i.e., $\Theta_i$'s) and may not be ideal in large scale denoising problems. This naturally leads us to the field of OT in search of strategies to correct for this deformation.

The area of OT has seen rapid growth in the past years with various applications in statistics and machine learning. In statistical theory, OT appears in at least two general settings: (i) as an interesting estimation problem in its own right, where one uses observations to either approximate the Wasserstein distance between two ground truth distributions (see e.g.,~\cite{Fournier-Guillin-2015, Weed-Bach-2019, delBarrio-Loubes-2019}) or to estimate the actual OT map between them (see e.g., \cite{Hutter-Rigollet-2021, Deb-Et-Al-2021, Divol-2022, Manole-Weed-2024, Fan-2023}), and (ii) as a tool to propose and/or analyze statistical models in classical Euclidean settings (e.g., as in~\cite{Hallin-2017, Hallin-2021, Hallin-2022, Ghosal-Sen-2022, Hallin-2023, Deb-Sen-2023}) or in more abstract settings where data sets consist of, for example, probability distributions (e.g., as in the regression setting for distribution-on-distribution data explored in works like \cite{Ghodrati2022,WassRegression,PanaretosEtAlMulti}). This paper better fits the second class of works, but the adaptation of our ideas to the finite data setting will require the exploration of questions that fall in the first category mentioned above. 

Connections between the denoising problem (understood in a general sense) and ideas from computational OT have been explored before in applications to image and signal denoising in works like \cite{DittmerEtAl,BlauMichaeli,WangWen,NEURIPS2021_d77e6859}; a hard constraint version of \eqref{eq:EquivProblemWithPnelizatin} has been considered in \cite{NEURIPS2021_d77e6859}, where a quantitative form of the so called distortion-perception tradeoff is established. Modern approaches for noise removal with additional good perception quality constraints have been proposed in \cite{delbracio2023inversion}. These approaches take advantage of the gradient structure that denoisers often have. In this paper, we pursue a deeper mathematical analysis than previous works in the literature and explore new approaches, motivated by ideas from the theory of OT, for recovering the OT-based denoiser. \nc One of the key tools that we use in our paper from the literature of OT is the concept of multimarginal OT (see \cite{Pass2015}), which has been explored in the past in a variety of fields including density functional theory in physics and chemistry \cite{Cotar2012,Buttazzo2012}, economics \cite{Ekeland2004,Carlier2008}, and image processing \cite{MOTImaging}, among others. This paper introduces new applications of closely related OT problems.

\subsection{Outline}
The rest of the paper is organized as follows. In Section~\ref{sec:DenoisingLatentPenalty} we present our main results on problems~\eqref{eq:Oracle-Bayes-risk} and~\eqref{eq:EquivProblemWithPnelizatin}. First, in Section \ref{sec:Prelim} we introduce some necessary notation and background that we use in the rest of the paper. Then, in Section \ref{sec:SqError} we state our first main result, Theorem \ref{thm:Main1}, which establishes the existence and uniqueness of solutions of \eqref{eq:Oracle-Bayes-risk}. In Section \ref{sec:SoftPenalty} we state Theorem \ref{thm:Main2}, which characterizes the solution of the soft penalty problem \eqref{eq:EquivProblemWithPnelizatin}. Section \ref{sec:DenoisingObservablePenalty} is devoted to the observable space penalization problem \eqref{eq:FModel}. First, we provide a characterization of the Frech\'et derivative of its objective under suitable differentiability assumptions on the likelihood model. Then we present Theorems \ref{cor:StructureSolutions} and \ref{thm:Soft-F}, where, respectively, we state the existence of solutions to \eqref{eq:FModel} and characterize the behavior of solutions to \eqref{eq:FModel} as the parameter $\tau$ goes to zero. In particular, Theorem \ref{thm:Soft-F} states that, under suitable identifiability assumptions, the OT-based denoiser $\delta^*$ can be recovered as a limit of solutions to \eqref{eq:FModel}. Sections \ref{sec:Proofs} and \ref{sec:DenoisingObservablePenaltyProofs} are devoted to the proofs of our main results from Sections \ref{sec:DenoisingLatentPenalty} and \ref{sec:DenoisingObservablePenalty}, respectively. In Section \ref{sec:Discussion}, the conclusions section, we discuss some future directions for research stemming from this work.  

In Appendices~\ref{sec:Exp-Family}-\ref{sec:EmpiricalBayes} we provide various discussions connected to our main results. In particular, in Appendices~\ref{sec:Exp-Family}-\ref{sec:Tweedie} we introduce exponential families of distributions and describe Tweedie's formula. In Appendix~\ref{app:MeasureFunctional} we state and prove a few results from measure theory and functional analysis that are relevant to the proof of Theorem~\ref{thm:Soft-F}. Appendix~\ref{sec:General-cost} briefly describes OT formulations of~\eqref{eq:Oracle-Bayes-risk} when using more general loss functions (beyond the squared error loss). In Appendix~\ref{sec:EmpiricalBayes} we briefly review the (nonparametric) maximum likelihood estimator of $G^*$, which could potentially be used to implement \eqref{eq:Oracle-Bayes-risk} in practical settings.

\section{Denoising with latent space penalization}
\label{sec:DenoisingLatentPenalty}

\subsection{Preliminaries}
\label{sec:Prelim}
We first introduce some definitions and notation from the theory of OT (see e.g.,~\cite{Villani2003, Villani}). For any metric space $\mathcal{X}$, let $\mathfrak{B}(\mathcal{X})$ denote the set of all Borel measurable subsets of $\X$, and let $\mathcal{P}(\X)$ be the set of all Borel probability measures over $\mathcal{X}$. It will be convenient to first introduce the notion of {\it pushforward} of a measure by a map and rewrite the constraint in \eqref{eq:Oracle-Bayes-risk} in terms of pushforwards.

\begin{defn}[Pushforward of a measure]\label{defn:Pushforward}
	Given a measurable map $\delta: \X \rightarrow \Y$ and a probability measure $\nu$ over $\X$, the measure $\delta_{\sharp} \nu$, the pushforward of $\nu$ by $\delta$, is the measure defined according to $\delta_{\sharp}\nu(A) := \nu(\delta^{-1}(A))$ for every Borel subset $A$ of $\Y$. In other words, if $X \sim \nu$, then $\delta(X) \sim \delta_{\sharp} \nu$.
\end{defn}

\begin{remark}
	\label{rem:PushForward}
	The constraint $\delta(Z) \sim G^*$ in \eqref{eq:Oracle-Bayes-risk} can be rewritten as $\delta_{\sharp}\mu =G^*$.    
\end{remark}

Let two probability measures $\nu, \tilde \nu$ be defined over two Polish spaces $\X$ and $\Y$, and consider a lower semicontinuous cost function $c: \X \times \Y \to [0, \infty]$. The dual of the Kantorovich OT problem (see e.g.,~\cite{Villani2003, Villani})
\begin{equation}
C(\nu, \tilde \nu):= \min_{\pi \in \Gamma(\nu, \tilde \nu)} \int \int c(x,y) \,d \pi(x,y), 
\label{eqn:OTGeneric}
\end{equation}
where $\Gamma(\nu, \tilde \nu)$ denotes the set of all Borel probability measures on the product $\X \times \Y$ with marginals $\nu$ and $\tilde \nu$ (a.k.a.~couplings between $\nu$ and $\tilde \nu$), 
is the problem
\begin{equation}
\label{eq:Dual}
\sup_{\phi, \psi} \int \phi(x) \,d\nu(x)  + \int \psi(y) \,d\tilde \nu (y), \;\;  \text{s.t. } \phi(x) + \psi(y) \leq c(x,y), \;\; \nu\text{-a.e.~}x \in \X, \; \tilde \nu\text{-a.e.~} y \in \Y,
\end{equation}
where $\phi$ and $\psi$ are, respectively, in $\mathbb{L}^1(\X,\nu)$ and $\mathbb{L}^1(\Y,\tilde \nu)$. Theorem 5.10 in \cite{Villani} guarantees that primal and dual problems are equivalent. Any solution pair $(\phi, \psi)$ of \eqref{eq:Dual}, if it exists, will be referred to as \textit{optimal dual potentials} for the OT problem \eqref{eqn:OTGeneric}. 

We will often consider the setting where the space $\X$ is a subset of some Euclidean space, $\X = \Y$, and $c(x,y)=|x-y|^2$. When in this setting, we will refer to \eqref{eqn:OTGeneric} as the $2$-OT problem between $\nu$ and $\tilde \nu$ and denote by $W_2^2(\nu, \tilde \nu)$ the minimum value in \eqref{eqn:OTGeneric}, which is nothing but the square of the so-called Wasserstein distance between $\nu$ and $\tilde \nu$. 

\begin{defn}[$2$-Wasserstein distance]
	Given two probability measures $\nu, \tilde \nu$ over $\R^p$ with finite second moments, we define their Wasserstein distance 
	$W_2(\nu, \tilde \nu)$ as
	\begin{equation}
	W_2^2(\nu, \tilde \nu):= \min_{\pi \in \Gamma(\nu, \tilde \nu)} \int |x - y|^2 \, d \pi(x,y).
	\label{def:2Wass}
	\end{equation}
\end{defn}
A landmark result in the theory of OT due to Brenier characterizes the optimal coupling $\pi$ for the $2$-OT problem between two measures $\nu$ and $\tilde \nu$ when $\nu$ is absolutely continuous w.r.t.~the Lebesgue measure; see e.g.,~\cite[Theorem 3.15]{Villani}.

% which characterizes solutions to optimal transport problems when the cost is the squared Euclidean distance; see \cite{Brenier} and Chapter 1.3 in \cite{Santambrogio}.

\begin{theorem}[Brenier] Let $\nu$ and $\tilde \nu$ be two Borel probability measures over $\R^p$ such that $\int |x|^2 \, d\nu(x) < \infty$ and $\int |y|^2 \, d{\tilde \nu}(y) < \infty$. Suppose further that $\nu$ has a Lebesgue density. Then there exists a convex function $\psi: \R^p \to \R \cup \{+\infty\}$ whose gradient $T= \nabla \psi$ pushes $\nu$ forward to $\tilde \nu$. In fact, there exists only one such $T$ that arises as the gradient of a convex function, i.e., $T$ is unique $\nu$-a.e. Moreover, $T$ uniquely minimizes Monge's problem:
	\[  \inf_{T  \::\:  T_{\sharp} \nu = \tilde \nu} \int | x- T (x) |^2 d \nu(x)  \]
	and the coupling $(\mathrm{Id} \times T)_{\sharp}\nu$ uniquely minimizes \eqref{def:2Wass}. In the above and in the remainder of the paper, the map $(\mathrm{Id} \times T) : \R^p \to  \R^p \times \R^p$ is defined as $(\mathrm{Id} \times T)(x) = (x, T(x))$. 
	\label{thm:Brenier}
\end{theorem}

\nc

\subsection{Rewriting \eqref{eq:Oracle-Bayes-risk} as an optimal transport problem}
\label{sec:SqError}
In this subsection we study problem~\eqref{eq:Oracle-Bayes-risk} and develop its connection with standard Monge and Kantorovich OT problems with a suitable cost function. Thanks to Remark \ref{rem:PushForward}, problem~\eqref{eq:Oracle-Bayes-risk} can be written as
\begin{equation}
\min_{\delta \: : \: \delta_{\sharp } \mu = G^* } \E_{(Z, \Theta) \sim P_{Z, \Theta}} \left[ | \delta(Z) - \Theta|^2  \right].    
\label{eq:Original}
\end{equation}
In turn, problem \eqref{eq:Original} is equivalent to:
\begin{equation}
\label{eq:EquivProblem}
\min_{\delta \: : \: \delta_{\sharp } \mu = G^* } \E_{Z  \sim \mu } \left[ | \delta(Z) - \overline{\theta}(Z)|^2  \right],    
\end{equation}
where 
\begin{equation}\label{eq:Post-Mean}
\overline{\theta}(z) := \E_{(Z, \Theta) \sim P_{Z, \Theta}} \left[ \Theta \mid Z =z  \right]
\end{equation} 
is the {\it posterior mean} (and the Bayes estimator under the quadratic loss). This equivalence follows from the well-known bias-variance decomposition for the squared error loss:
\[ \E[ |\delta(Z) - \Theta|^2  \mid Z ] =  \E[|\delta(Z) - \overline{\theta}(Z)|^2 \mid Z  ] + \E[|  \overline{\theta}(Z) - \Theta|^2 \mid Z  ],  \]
which implies that for any arbitrary $\delta: {\R^d} \rightarrow \R^m$ we have
\[ \E_{(Z, \Theta) \sim P_{Z, \Theta}} \left[ | \delta(Z) - \Theta|^2  \right] = \E_{Z \sim \mu}[|\delta(Z) - \overline{\theta}(Z)|^2 ] + \E_{(Z, \Theta)\sim P_{Z,\Theta}}[|  \overline{\theta}(Z) - \Theta|^2 ],\]
from where it follows that the objective in \eqref{eq:Original} is equal to the objective function in \eqref{eq:EquivProblem} up to the constant $R_{\text{Bayes}}:=\E_{(Z, \Theta)\sim P_{Z,\Theta}}[|  \overline{\theta}(Z) - \Theta|^2 ] $, i.e., the Bayes risk.

The advantage of problem \eqref{eq:EquivProblem} is that, as discussed below, it is amenable to the type of relaxation methods that have been studied in OT theory. Indeed, in order to construct a solution to \eqref{eq:EquivProblem} (and thus also to \eqref{eq:Original}), at least for certain families of models $( \{ p(\cdot \mid \theta) \}_{\theta \in \Omega}, G^*)$ satisfying suitable assumptions, we will first consider a Kantorovich relaxation of \eqref{eq:EquivProblem} given by
\begin{equation}
\min_{\pi \in \Gamma(\mu, G^*)}  \iint c_{G^*}(z, \vartheta ) \, d\pi(z,\vartheta),
\label{eqn:Kantorovich}
\end{equation} 
where the  
% and propose a computational approach to construct sensible oracle estimators $\delta$ motivated by problem \eqref{eq:EquivProblem},
cost function $c_{G^*}(\cdot,\cdot)$ is defined as
\begin{equation}\label{eq:c_z_y} 
c_{G^*}(z, \vartheta) := | \overline{\theta}(z) - \vartheta |^2, \quad (z, \vartheta)\in {\R^d} \times \Omega;
\end{equation}
note the dependence of $G^*$ on the cost function $c_{G^*}(z, \vartheta)$ in~\eqref{eq:c_z_y} via the Bayes estimator $\bar \theta(z)$, which depends on $G^*$. 

\nc

We make the following assumptions.

\begin{assumption}
	\label{assump:SecondMomentsG}
	The distribution $G^*$ is such that $\int_{\Omega} |\vartheta|^2 dG^*(\vartheta) <\infty $, i.e., $G^*$ has finite second moments. 
\end{assumption}

\begin{assumption}
	\label{assump:PosteriorMean}
	The measure $\overline \theta_\sharp \mu$ is absolutely continuous w.r.t.~the Lebesgue measure in $\R^m$. 
\end{assumption}

\begin{remark}[On our assumptions] Since $\mu$ is absolutely continuous w.r.t.~the Lebesgue measure (recall \eqref{eq:Marg-Dens}), note that Assumption~\ref{assump:PosteriorMean} holds if we assume that the map $\overline{\theta}:\R^d \rightarrow \R^m$ is locally-Lipschitz (and thus differentiable Lebesgue a.e.) and that the Jacobian matrix $D \overline{\theta}(z) \in \R^{d\times m}$ has full rank $\mu$-a.e.~$z$; indeed, this implication follows from the so called coarea formula (see e.g., the theorem in Section 3.1 in~\cite{Federer}). Thus, implicitly, we would be assuming that $d \ge m$. In particular, the above is satisfied for the following scenario. If $\{p(\cdot \mid \theta)\}_{\theta \in \Omega}$ is a regular $k$-parameter exponential family in canonical form, then $\overline{\theta}(\cdot)$ is the gradient of a convex function $\kappa(\cdot)$ (which happens to be the log-partition function of the family); see~\cref{sec:Exp-Family} and~\cref{sec:Tweedie}. Moreover, $\kappa(z)$ is a strictly convex function of $z$ on its domain if the representation is {\it minimal}; see e.g.,~\cite[Proposition 3.1]{WJ-2008}. As convex functions are a.e.~twice continuously differentiable, the Jacobian matrix $D \overline{\theta}(z) \in \R^{d\times m}$ exists a.e., and thus Assumption~\ref{assump:PosteriorMean} is automatically satisfied.~Assumption~\ref{assump:SecondMomentsG} just assumes a finite second moment condition on $G^*$, which is quite mild.
\end{remark}

We are ready to state our first main result.

% \begin{theorem}
% \label{thm:Main1}
% Suppose that $\mu$ has a density with respect to the Lebesgue measure and that Assumptions \ref{assump:PosteriorMean} and \ref{assump:SecondMomentsG} hold. Then there exists a unique solution $\pi^*$ to \eqref{eqn:Kantorovich}. Moreover, the solution $\pi^*$ has the form:
% \[\pi^*=(Id \times \delta^*)_\sharp \mu \]
% for a map $\delta^*$ that is a solution to the original problem \eqref{eq:Original}. In particular, the barycentric oracle estimator $\delta_{\pi^*}$, which is equal to $\delta^*$, is a solution to \eqref{eq:Original}. 
% \end{theorem}
\nc

\begin{theorem}
	\label{thm:Main1}
	Under Assumptions \ref{assump:SecondMomentsG} and \ref{assump:PosteriorMean}, there exists a unique solution $\pi^*$ to problem \eqref{eqn:Kantorovich} with cost function~\eqref{eq:c_z_y}, which takes the form
	\[\pi^*=(\mathrm{Id} \times \delta^*)_\sharp \mu \]
	for a map $\delta^*(\cdot)$ that is the $\mu$-a.e. unique solution to problem \eqref{eq:Original}, i.e., it is the OT-based denoiser. Furthermore, $\delta^*(\cdot)$ can be written as
	\begin{equation}\label{eq:delta-OT}
	\delta^*(z)= \nabla \varphi ^* ( \overline{\theta}(z)), \qquad \mbox{for }\;\;z \in {\R^d},
	\end{equation}
	to be read: ``the gradient of the function $\varphi^*$ evaluated at $\overline{\theta}(z)$", where $\varphi^*: \R^m \rightarrow \R \cup \{+\infty\}$ is a convex function. In fact, $T^*:=\nabla \varphi^*$ is the solution to the standard quadratic cost Monge OT problem
	\begin{equation}
	\min_{ T:  T_{\sharp }({\overline{\theta}}_{\sharp} \mu)=G^*} \int |\theta- T(\theta) |^2 \,d \overline{\theta}_{\sharp} \mu (\theta)
	\label{eq:OT_PushForward}
	\end{equation} 
	between the measures $\overline{\theta}_{\sharp} \mu$ and $G^*$.
\end{theorem}

The proof of Theorem \ref{thm:Main1} is presented in Section~\ref{sec:Proofs}. It builds upon Brenier's theorem (Theorem \ref{thm:Brenier}). The first part of Theorem~\ref{thm:Main1} implies that, under Assumptions~\ref{assump:SecondMomentsG} and~\ref{assump:PosteriorMean}, the value of Kantorovich's relaxation  problem in~\eqref{eqn:Kantorovich} is indeed the same as that of Monge's problem~\eqref{eq:Original}. Further, the optimal coupling in~\eqref{eqn:Kantorovich} yields the  solution to~\eqref{eq:Original}. Theorem~\ref{thm:Main1} further says that the optimal solution $\delta^*(\cdot)$ of~\eqref{eq:Original} is related to the Bayes estimator~\eqref{eq:Post-Mean}; in fact, $\delta^*(\cdot)$ pushes the Bayes estimator $\bar{\theta}(\cdot)$ to satisfy the distributional constraint $\delta^*(Z) \sim G^*$. The fact that $\delta^*(\cdot)$ has such a simple form is not immediately obvious from the original formulation of the problem in \eqref{eq:Oracle-Bayes-risk}.

\begin{remark}[Normal-normal location model]\label{rem:Nor-Nor}
	Suppose that $d = m$ and $\Theta \sim N_m(\theta^*,\Sigma^*)$ and $Z \mid  \Theta = \theta \sim N_d(\theta,\Sigma)$, where $\theta^* \in \R^m$ is known, and $\Sigma^* \in \R^{m \times m}$ and $\Sigma\in \R^{d \times d}$ are symmetric positive definite (fixed) matrices. It is then well-known that $$ \bar \theta(Z) = \Sigma^* (\Sigma^* + \Sigma)^{-1} Z +  \Sigma (\Sigma^* + \Sigma)^{-1} \theta^* $$ which shows that
	\begin{equation}\label{eq:Normal-Normal}
	\bar \theta(Z) \sim N_m \Big( \theta^*, A \Big), \qquad \mbox{where } \;\; A :=\Sigma^* (\Sigma^* + \Sigma)^{-1} \Sigma^*,
	\end{equation}
	as unconditionally, $Z \sim N_d(\theta^*, \Sigma^* + \Sigma)$. Therefore, by Theorem~\ref{thm:Main1}, to find the OT-based denoiser $\delta^*$ we need to find the OT map $T^*$ between the distributions $N_m \Big( \theta^*, A \Big)$ and $N_m \Big( \theta^*, \Sigma^* \Big)$ which is given by
	$$T^*: y \mapsto \theta^* + B(y - \theta^*) \qquad \mbox{where } \;\;  B := A^{-1/2} (A^{1/2} \Sigma^* A^{1/2})^{1/2} A^{-1/2}  .$$
	Thus, the OT-based denoiser $\delta^*$ has the form $\delta^*(Z) = T^*(\bar \theta(Z))$. 
	
	To get a better feel for the estimators --- $\bar \theta(Z)$ and $\delta^*(Z)$ --- in this problem, let us consider the special case $d=m =1$ with $\Sigma = 1$ and $\Sigma^* = \tau^2$ and $\theta^* = 0$. Here we can see that the Bayes estimator satisfies $$\bar \theta(Z) := \tau^2 (1 + \tau^2)^{-1} Z, \qquad \mathrm{and \;\; thus}, \qquad \bar \theta(Z) \sim N(0,\tau^4/(1 + \tau^2)).$$ Note that $\tau^4/(1 + \tau^2) < \tau^2$ and thus the Bayes estimator has lower variance than $G^* \equiv N(0,\tau^2)$ (see Figure~\ref{fig:Motivation} for an illustration of this phenomenon via a simple simulation). However, the OT-based denoiser $\delta^*$ has the form $T^*(\bar \theta(Z))$ where $T^*(y) := \tau (\tau^4/(1 + \tau^2))^{-1/2} y$. Here the Bayes risk (i.e., $\E[(\bar \theta(Z) - \Theta)^2]$) is $\tau^2/(1 + \tau^2) < 1$ and the risk of $\delta^*$ is $2 \tau^2 \left(1 - \frac{\tau}{1 + \tau^2} \right)$ (see the black and {\color{orange} orange} curves in the right panel of Figure~\ref{fig:Motivation}).
\end{remark}

\begin{remark}[When $m = 1$] In the special case when $m=1$, the OT-based denoiser $\delta^*(\cdot)$ in~\eqref{eq:delta-OT} can be explicitly expressed as $
	\delta^*(z)= F^{-1}_{G^*} (F_{\overline{\theta}} ( \overline{\theta}(z)))$, for $z \in \R$, where $F^{-1}_{G^*}$ is the quantile function corresponding to the distribution $G^*$ (i.e., $F^{-1}_{G^*}(p) := \inf \{x \in \R: p \le F_{G^*}(x) \}$, for $p \in (0,1)$) and $F_{\overline{\theta}}$ is the distribution function of the random variable $\overline{\theta}(Z)$. This follows easily from the fact that, in one-dimension, Brenier maps have explicit solutions in terms of distribution/quantile functions.   
\end{remark}

%{\color{red} TO DO: Consider the model where $Z_i \mid \Theta_i \sim N(\Theta_i,1)$ and $\Theta_1,\ldots, \Theta_n$ i.i.d.~$N(0,\tau^2)$. Find the oracle denoizer~\eqref{eq:delta-OT}. Compute it's Bayes risk and compare that to the optimal Bayes risk.}

%Estimate $\tau^2$ by Empirical Bayes methods to obtain an estimator of the oracle denoizer. Study it's frequentist risk behavior and compare with that of the MLE, the James-Stein estimator, etc. Is this estimator admissible?

\subsection{Soft penalty versions of \eqref{eq:Original}}
\label{sec:SoftPenalty}
We now consider problem~\eqref{eq:EquivProblemWithPnelizatin},
%\eqref{\begin{equation}
% \label{eq:EquivProblemWithPnelizatin2}
%  \inf_{\delta: \R^d \rightarrow \R^m } \E_{(Z, \Theta)\sim P_{Z,\Theta} } \left[ | \delta(Z) - \Theta|^2  \right] + \frac{1}{2 \tau} W_2^2(\delta_{\sharp} \mu, G^*) , 
%  \end{equation}}
which is a type of relaxation of problem \eqref{eq:Original} where we use a soft penalty on $\delta$ to enforce $\delta_{\sharp}\mu$ to be sufficiently close to $G^*$ as opposed to enforcing a hard constraint as in  \eqref{eq:Original}. The strength of the penalization is determined by the parameter $\tau$, and, intuitively, we should expect to recover the classical Bayes estimator $\bar \theta(Z)$ when $\tau \rightarrow \infty$, and the OT-based denoiser $\delta^*(Z)$ when $\tau \rightarrow 0$. As we show in the result below (see Section~\ref{sec:Interpolation} for its proof), the estimators recovered by solving \eqref{eq:EquivProblemWithPnelizatin} are simple linear interpolators of the Bayes estimator $\bar \theta(Z)$ and $\delta^*(Z)$. 
\begin{theorem}\label{thm:Main2}
	Under the same assumptions as in Theorem \ref{thm:Main1}, there exists a unique solution $\delta^*_\tau$ to \eqref{eq:EquivProblemWithPnelizatin}. Furthermore, the map $\delta_\tau^*(\cdot)$ can be written as
	\begin{equation}\label{eq:delta-OT-tau}
	\delta_\tau^*(z)= \frac{2\tau}{1 +2\tau}\overline{\theta}(z) + \frac{1}{1+2\tau} \delta^*(z)
	, \qquad \mbox{for }\;\;z \in {\R^d},
	\end{equation}
	where $\delta^*(\cdot)$ is the map from Theorem \ref{thm:Main1}.
\end{theorem}

\begin{remark}[On the proof of Theorem~\ref{thm:Main2}]
	The proof of Theorem \ref{thm:Main2} is based on a simple relaxation argument that mimics the relaxation in \cite{AguehCarlier} used to reformulate Wasserstein barycenter problems as multimarginal OT problems. 
\end{remark}

% We remark that the above simple form of the solution to~\eqref{eq:EquivProblemWithPnelizatin} is due to the fact that in \eqref{eq:EquivProblemWithPnelizatin} we use the $2$-Wasserstein distance to penalize the discrepancy between measures. {\color{red} In general, one may consider other distances to penalize the discrepancy between $\delta_{\sharp}\mu$ and $G^*$}.

% {\color{red}
% \subsection{Some examples}
% Consider the normal location mixture introduced in Example~\ref{ex:Norm-loc-Mix}. In this case the posterior mean $\overline{\theta}(z)$ (see~\eqref{eq:Post-Mean}) is given by 
% \begin{equation}\label{eq:Tweedie-Std-Nor}
% \overline{\theta}(z) = 
% \E[\Theta \mid Z = z] = z +  \frac{\nabla f_{G^*}(z)}{f_{G^*}(z)},\end{equation}
% using Tweedie's formula; see e.g.,~\cite{Efron11}. 
% Note that in this case, $\mu$, the marginal density of $Z$ is absolutely continuous and $\overline{\theta}(z)$ is locally Lipschitz. Thus, recalling the definition of $c(z,y)$ from~\eqref{eq:c_z_y}, we have $\nabla_z c(z,y) = 2D_z \overline{\theta}(z) (\overline{\theta}(z) -y) $, where $$ D_z \overline{\theta}(z)  = I_d + \frac{\sigma^2}{f_{G^*}(z)}\nabla^2 f_{G^*}(z) - \frac{\sigma^2}{f_{G^*}^2(z)}\nabla f_{G^*}(z) \nabla f_{G^*}(z)^\top.$$

% Note that the posterior mean can be expressed in terms of the gradient of the log-likelihood in a variety of models, including for exponential families; see Appendix~\ref{sec:Tweedie} (in particular, see~\eqref{eq:Tweedie}). }

% \section{$F$-modeling}

\section{Denoising with observable space penalization}
\label{sec:DenoisingObservablePenalty}

% Although the approach outlined in~\eqref{eq:Original} has a nice solution with many appealing properties (see e.g., Theorems~\ref{thm:Main1} and~\ref{thm:Main2}), implementing it in practice when $G^*$ is unknown can be a nontrivial task. This is because, to develop a sample analogue to~\eqref{eq:Original} we first need to estimate $G^*$, which is sometimes called the deconvolution problem in statistics, and for many smooth families (e.g., when $p(\cdot \mid \theta)$ is the normal distribution with mean $\theta$) it is well-known that estimation of $G^*$ is a hard statistical problem for which the minimax rates of estimation can be logarithmic (see e.g.,~\cite{Fan, Zhang, }). {\color{red} This motivates the need to develop a ...}

% In contrast to the $G$-modelling problem from previous sections, the objective in the $F$-model \eqref{eq:FModel} is Fr\'echet differentiable, with a simple expression for its gradient. This is the content of the next theorem ({\color{red}see Section ?? for a proof}). 

Although the characterization of the OT-based denoiser $\delta^*(\cdot)$ as a solution to an OT problem is appealing, in most real applications $G^*$ is unknown. As discussed in the Introduction right before \eqref{eq:FModel} (see also Appendix \ref{sec:EmpiricalBayes}), one possible approach to go around this issue is to estimate $G^*$ using i.i.d.~data from~\eqref{eq:Mix-Mdl} using deconvolution techniques that are also used in the empirical Bayes literature; the resulting approach is in line with the concept of $g$-modelling discussed in \cite{Efron-2014}. In this section, and in the spirit of the $f$-modelling discussed in \cite{Efron-2014}, \nc we take a different approach and study yet another formulation for the denoising problem that closely resembles \eqref{eq:EquivProblemWithPnelizatin} but where we directly work with $\mu$, the (marginal) distribution of the observed data (see~\eqref{eq:Marg-Dens}). In particular, we consider the optimization problem~\eqref{eq:FModel} with objective $\mathcal{E}_{\tau}(\delta)$.

First, we provide an explicit formula for the Gateaux derivative of $\mathcal{E}_\tau$ w.r.t.~$\delta$, when the likelihood model is sufficiently regular. In principle, this Gateaux derivative can be used to implement a first order optimization method to find solutions of \eqref{eq:FModel}, but as discussed in the Introduction, the convergence to global optimizers of this scheme cannot be guaranteed due to the non-convexity of $\mathcal{E}_\tau$. For this reason we consider an alternative methodology which holds under milder assumptions and which will allow us to: (1) prove the existence of solutions of \eqref{eq:FModel}, (2) suggest a linear optimization problem for solving \eqref{eq:FModel}, and (3) recover $\delta^*$, the OT-based denoiser, without explicit knowledge of $G^*$.  Throughout this section we make the following assumption, which is used to guarantee that problem \eqref{eq:FModel} is non-trivial.

\begin{assumption}
	\label{assump:SecondMomentsmu}
The marginal distribution $\mu$ with density as in \eqref{eq:Marg-Dens} has finite second moments.
\end{assumption}
\nc

\begin{prop}
	\label{thm:FrechetDiff}
	Suppose that Assumptions \ref{assump:SecondMomentsG} and \ref{assump:SecondMomentsmu} hold. Suppose also that the likelihood model is such that $p(z \mid \theta)$ is continuously differentiable in $\theta$ (for every $z \in \R^d$). Let $\delta \in \mathbb{L}^2(\R^d :  \R^m ; \mu)$ and suppose that $\mu, \mu_\delta$ (recall $\mu_\delta$ was defined in \eqref{eq:mu-delta}) are such that they admit a unique (up to constant shifts) solution $(\tilde \phi, \tilde \psi)$ to the dual of the $2$-OT problem between $\mu$ and $\mu_\delta$ \nc; in particular, 
	\[  \int \tilde \phi \, d\mu + \int \tilde \psi \, d \mu_\delta = W_2^2(\mu, \mu_\delta). \] 
	Finally, suppose that the function
	\[z \in \R^d \mapsto \int_{\R^d} \tilde \psi(z')  
	\nabla_\theta p(z' \mid \delta(z )) \, dz'  \] 
	belongs to $L^2(\R^d: \R^m; \mu)$.

	Then the objective function $\mathcal{E}_\tau: \mathbb{L}^2({\R^d}: \R^m;\mu) \rightarrow \R$ defined in \eqref{eq:FModel} is Gateaux differentiable at $\delta$, and its gradient at that point takes the form:
	\begin{equation}
	\nabla \mathcal{E}_\tau(\delta)  = 2(\delta(\cdot ) - \overline{\theta}(\cdot))  + \frac{1}{2\tau}  \int_{\R^d} \tilde \psi(z') 
	\nabla_\theta p(z' \mid \delta(\cdot )) \,dz'.
	\label{eqn:GradientE}
	\end{equation}
\end{prop}
\begin{proof}
	
	Given the form of $\mathcal{E}_\tau$, it suffices to compute the Gateaux derivative of $W^2_2(\mu, \mu_\delta)$ at $\delta$. Let $\eta \in \mathbb{L}^2({\R^d}: \R^m; \mu)$ be arbitrary. Taking the derivative of $W_2^2(\mu_{\delta + \epsilon \eta} , \mu)$ w.r.t.~$\epsilon$ at $\epsilon=0$, we obtain
	\begin{align}
	\begin{split}
	\frac{d}{d\eps} \Big |_{\eps = 0} W_2^2 (\mu_{\delta + \eps \eta } , \mu ) \, &  =     \frac{d}{d\eps} \Big |_{\eps = 0} \sup_{(\phi, \psi) \textrm{ s.t. } \phi(x)+\psi(y) \leq |x-y|^2} \int \phi \,d\mu  + \int \psi  \,d\mu_{\delta + \eps \eta}  
	\\& = \frac{d}{d\eps} \Big |_{\eps = 0} \int \tilde \psi \,d\mu_{\delta + \eps \eta}
	\\& = \frac{d}{d\eps} \Big |_{\eps = 0}  \int \int \tilde \psi(z')p(z' \mid \delta(z) + \epsilon \eta (z)) \,dz' \,d\mu(z)  
	\\& =  \int_{\R^d} \int_{\R^d} \tilde \psi(z') \frac{d}{d \eps} \Big |_{\eps = 0} (p(z'\mid \delta(z)+ \eps \eta(z))) \,dz' \, d\mu(z) 
	\\& = \int_{\R^d} \int_{\R^d} \tilde \psi(z') 
	\langle \nabla_\theta p(z' \mid \delta(z)),  \eta(z) \rangle_{\R^m} \,dz' \,d\mu(z)
	\\& = \int_{\R^d} \left\langle  \eta(z) ,\int_{\R^d} \tilde \psi(z') 
	\nabla_\theta p(z' \mid \delta(z))   \,dz'  \right \rangle_{\R^m} \,d\mu(z).
	\end{split}
	\end{align}      
	The second equality follows as in Proposition 7.17 (and Proposition 7.18) in \cite{Santambrogio} and the third equality just uses the definition of $\mu_{\delta + \eps \eta}$. Since $\eta$ was arbitrary, we deduce \eqref{eqn:GradientE}. 
	% The above computation characterizes the first variation of the function $\delta \in \mathbb{L}^2({\R^d} : \R^m; \mu) \mapsto W^2(\mu_\delta, \mu)$ at $\delta$ and it proves that the Fr\'echet derivative of this function is given by 
	% \[  \int_{\R^d} \tilde \psi(z')  
	%   \nabla_\theta p(z' \mid \delta(\cdot )) \, dz'.  \]
	% Notice that the above is indeed a function from ${\R^d}$ into $\R^m$.
\end{proof}

\begin{remark}
	\label{rem:GradDescent}
	
	When given finitely many observations $Z_1, \dots, Z_n$ sampled from $\mu$, the formula in \eqref{eqn:GradientE} suggests the following algorithm to construct a (finite sample) denoising estimator from the observations. In what follows we let
	\[ \mu_n:= \frac{1}{n} \sum_{i=1}^n \delta_{Z_i}  \]
	be the empirical measure of the observations.

	Set $k=0$, and initialize $\delta_k(Z_1), \dots, \delta_k(Z_n) \in \R^m$.
	
	Then do until a stopping criterion is satisfied:
	\begin{enumerate}
		\item Find $\tilde \psi$, optimal dual potential for the 2-OT problem between $\mu_n$ and the measure with density  $\frac{1}{n} \sum_{i=1}^n p(\cdot \mid \delta_k(Z_i))$. 
		\item Set, for $i=1, \dots, n$, $$\delta_{k+1}(Z_i) := \delta_k(Z_i)  - \lambda \left(  2( \delta_k(Z_i) - \overline{\theta}(Z_i)  ) +\frac{1}{2\tau} \int  \tilde \psi(z') \frac{\nabla_\theta p(z' \mid\delta_k(Z_i))}{p(z' \mid\delta_k(Z_i))} p(z' \mid\delta_k(Z_i)) \, dz'\right).$$     
		\item Set $k=k+1$.
	\end{enumerate}
	In the above, $\lambda>0$ is a time step parameter. Note that when the likelihood model is an exponential family of distributions, we can use
	Tweedie's formula (see Appendix~\ref{sec:Tweedie}) to estimate $\overline{\theta}(Z_i)$. The computation of $\tilde \psi$ can be carried out with an OT solver. We leave it for future work to explore the use of different solvers for computing the gradient of $\mathcal{E_\tau}$ in practical finite data settings.   \nc 
	
\end{remark}

% {\color{red} Initialization? Convergence criterion? } 

% In the next theorem we discuss the behavior of solutions of \eqref{eq:FModel} as $\tau_\rightarrow 0$ under additional identifiability and regularity properties of the family $\{ p(\cdot|\theta) \}_{\theta \in \Theta}$  

\subsection{A Kantorovich relaxation of \eqref{eq:FModel} and recovery of $\delta^*$}
\label{sec:KantoRelax}

We now turn our attention to studying the existence of solutions to problem \eqref{eq:FModel}. To achieve this, we first introduce a suitable Kantorovich relaxation of \eqref{eq:FModel} for which we can prove existence of solutions using the direct method of the calculus of variations. Under Assumption \ref{assump:weak-cont} stated below, we will further characterize the structure of solutions of this relaxation and in particular show that any solution to \eqref{eqn:Relaxation} (see below) naturally induces a solution to the original problem \eqref{eq:FModel}. To define the desired Kantorovich relaxation, let us first introduce the set of admissible couplings 
\begin{equation}
\mathcal{A} := \left\{ \gamma \in \mathcal{P}({\R^d} \times \Omega \times {\R^d} \times {\R^d}) : \gamma_1 = \mu, \gamma_4 = \mu , \text{ and } \int p(\cdot \mid \theta) \,d\gamma_2(\theta)  = \gamma_3(\cdot) \right\};
\label{def:FeasibleSet}
\end{equation}
in the above display by $\gamma_k$, for $k=1,2,3,4$, we mean the $k$'th marginal of $\gamma$. We observe that the set $\A$ is determined by $\mu$ (the marginal distribution of observed variables) and $\{p(\cdot \mid \theta)\}_{\theta \in \Omega}$ (the likelihood model). We now introduce the problem:
\begin{equation}
\inf_{\gamma \in \A}  \int \cc_\tau(z_1, \theta, z_3, z_4) \, d\gamma(z_1, \theta, z_3, z_4), 
\label{eqn:Relaxation}
\end{equation}
where the cost function $\cc_\tau$ is defined as:
\begin{equation}
\cc_\tau(z_1, \theta, z_3, z_4) :=  | \theta - \overline{\theta} (z_1)|^2 + \frac{1}{2\tau} | z_3 - z_4 |^2. 
\label{eqn:CostRelaxation}
\end{equation}

\begin{remark}[Comparison with multimarginal OT]
	\label{rem:CompMOT}
	Problem \eqref{eqn:Relaxation} resembles a multimarginal optimal transport (MOT) problem (e.g., see \cite{Pass2015}) with four marginals, but differs from a standard MOT in the type of constraint that we put on the second and third marginals of the coupling $\gamma$.   
\end{remark}

We will make the following assumptions on our probabilistic model.

\begin{assumption}\label{assump:weak-cont} 
	We assume that the set $\Omega$ is a closed subset of $\R^m$. In addition, we assume that the family of probability measures $ \{ p(\cdot \mid \theta) \}_{\theta \in \Omega}$ is continuous in $\theta$ in the weak sense, i.e., if $\{\theta_n\}_{n \in \N}$ is a sequence in $\Omega$ converging to some $\theta \in \Omega$, then $p(\cdot \mid\theta_n)$ converges weakly to $p(\cdot \mid\theta)$. 
\end{assumption}

\begin{assumption}\label{assump:ContinuityThetaBar}
	We assume that the posterior mean $\overline{\theta}(z)$ is continuous for $\mu$-a.e.~$z \in \R^d$.
\end{assumption}

\begin{remark}[On the first part of Assumption \ref{assump:weak-cont}]
	In order to prove the existence of solutions to problem \eqref{eqn:Relaxation} we assume that $\Omega$ is a closed subset of $\R^m$ for simplicity. In case $\Omega$ is not closed, one can consider modifying the definition of problem \eqref{eqn:Relaxation} by changing all appearances of $\Omega$ with $\overline{\Omega}$, the closure of $\Omega$. This can be done if we assume that the family $\{ p(\cdot \mid \theta) \}_{\theta \in \Omega}$ can be extended to a family of distributions $\{ p(\cdot \mid \theta) \}_{\theta \in \overline{\Omega}}$ (not necessarily with densities w.r.t.~the Lebesgue measure) for which we still have the weak continuity property: if $\{\theta_n \}_n \subseteq  \overline{\Omega}$  and $\theta_n \rightarrow \theta$, then $p(\cdot \mid \theta_n)$ converges weakly to  $p(\cdot \mid \theta)$. For instance, this can be done in the normal scale mixture problem in Example \ref{ex:Norm-scale-Mix}.   
\end{remark}

\begin{remark}[On Assumption \ref{assump:ContinuityThetaBar}] We will also impose Assumption \ref{assump:ContinuityThetaBar} to guarantee the existence of solutions to the relaxation problem \eqref{eqn:Relaxation}. This assumption is mild and for example is satisfied when $\{ p(\cdot \mid \theta) \}_{\theta \in \Omega}$ is an exponential family of distributions under suitable assumptions (see Lemma~\ref{lem:Grad-Cvx} in Appendix~\ref{sec:Tweedie}). Indeed, in this case $\overline{\theta}(\cdot)$ coincides with the gradient of a real-valued convex function. As, by Alexandrov's theorem, a convex function is (Lebesgue) a.e.~twice differentiable, its gradient is (Lebesgue) a.e.~continuous. Since $\mu$ has a density w.r.t.~the Lebesgue measure, it then follows that $\overline \theta(z)$ is indeed continuous for $\mu$-a.e.~$z \in \R^d$. %Another instance where Assumption~\ref{assump:ContinuityThetaBar} can be easily checked is  
\end{remark}

As stated in the next theorem, problem \eqref{eqn:Relaxation} admits minimizers. More importantly, all minimizers of this problem possess a convenient structure that we later use to prove existence of solutions to problem \eqref{eq:FModel}.

\begin{theorem}
	\label{lem:StructureGammastar}
	Suppose Assumptions~\ref{assump:SecondMomentsG},~\ref{assump:SecondMomentsmu},~\ref{assump:weak-cont} and~\ref{assump:ContinuityThetaBar} hold. Then there exist solutions to \eqref{eqn:Relaxation}. Moreover, if $\gamma^*$ is a solution of \eqref{eqn:Relaxation}, then $\gamma_{12}^*$, the projection of $\gamma$ onto the first two coordinates, is a solution to the problem
	\[ \inf_{\pi \in \Gamma(\gamma_1^*, \gamma_2^* )} \int |\theta - \overline{\theta}(z)|^2 \, d\pi(z, \theta). \]
	In turn, under Assumption \ref{assump:PosteriorMean}, $\gamma_{12}^*$ must have the form $\gamma_{12}^* = ({\rm Id} \times  \delta_{\gamma^*})_{\sharp } \mu$ for $\delta_{\gamma^*} \in \mathbb{L}^2({\R^d} : \R^m; \mu)$ the unique solution to the problem: 
	\begin{equation}
	\label{eq:TransportProblemCharacterizingRelaxation}
	\inf_{\delta: \delta_{\sharp} \mu= \gamma_2^* } \int |\delta(z) - \overline{\theta}(z)|^2 \, d \mu(z).  
	\end{equation}
\end{theorem}

Theorem \ref{lem:StructureGammastar} is proved in Section~\ref{pf:Theo-3.2}, and, as stated earlier, will be used to deduce the existence of solutions of \eqref{eq:FModel}. Precisely, as we state in Theorem \ref{cor:StructureSolutions} below, the existence of solutions of \eqref{eq:FModel} follows from the equivalence between problems \eqref{eq:FModel} and \eqref{eqn:Relaxation}. To describe this equivalence, we introduce some notation first.

Given $\delta \in \mathbb{L}^2({\R^d} : \R^m; \mu)$, let $\pi^\delta_{34}$ be a $2$-OT plan between $\mu_\delta $ (as defined in~\eqref{eq:mu-delta}) and $\mu$. Using $\pi^\delta_{34}$, we define $\gamma_\delta$ as the measure which acts on an arbitrary test function $\phi: {\R^d} \times \Omega \times {\R^d} \times {\R^d} \rightarrow \R$ according to
\begin{equation}
\label{eqn:Delta}
\int \phi (z_1, \theta, z_3, z_4) \, d \gamma_\delta (z_1,\theta, z_3,z_4) =  \int \phi (z_4, \delta(z_4), z_3, z_4) \, d \pi^\delta_{34}(z_3,z_4). 
\end{equation}
In simple terms, to sample from $\gamma_\delta$ it is sufficient to sample $(z_3, z_4) \sim \pi_{34}^\delta $ and then set $z_1=z_4$ and $\theta= \delta (z_4)$. Notice that $\gamma_\delta \in \A$. The proof of Theorem~\ref{cor:StructureSolutions} below can be found in Section~\ref{pf:Theo-3.3}.

\begin{theorem}
	\label{cor:StructureSolutions}
	Under Assumptions \ref{assump:SecondMomentsG},  \ref{assump:PosteriorMean}, \ref{assump:SecondMomentsmu} ,\ref{assump:weak-cont}, and \ref{assump:ContinuityThetaBar} the following properties hold:
	\begin{enumerate}
		\item Let $\gamma^*$ be any solution to \eqref{eqn:Relaxation}. Then the map $\delta_{\gamma^*}$ for which $\gamma_{12}^* = ({\rm Id} \times  \delta_{\gamma^*})_{\sharp } \mu$ is a solution to \eqref{eq:FModel}. In particular, thanks to Theorem \ref{lem:StructureGammastar}, there exist solutions to \eqref{eq:FModel}.
		\item Conversely, if $\tilde \delta$ is a solution to \eqref{eq:FModel}, then $\gamma_{\tilde \delta} \in \mathcal{P}({\R^d} \times \Omega \times {\R^d} \times {\R^d}) $ defined as in \eqref{eqn:Delta} for $\delta=\tilde \delta$ is a solution to \eqref{eqn:Relaxation}.
		
	\end{enumerate}
\end{theorem}

\begin{remark}[Equivalence between~\eqref{eqn:Relaxation} and~\eqref{eq:FModel}]
	Theorem \ref{cor:StructureSolutions} captures the equivalence between problems \eqref{eqn:Relaxation} and \eqref{eq:FModel}: from a solution $\gamma^*$ to \eqref{eqn:Relaxation} (which exists by the first part of Theorem \ref{lem:StructureGammastar}) we can obtain a map $\delta^*$ that is a solution to \eqref{eq:FModel}. Conversely, from a solution to \eqref{eq:FModel} we can construct a solution to \eqref{eqn:Relaxation}. 
	Interestingly, the relaxation \eqref{eqn:Relaxation} provides an avenue for designing alternative numerical methods for optimizing \eqref{eq:FModel} that do not rely on the gradient descent strategy in the $\mathbb{L}^2({\R^d} : \R^m; \mu)$ space suggested at the beginning of Section \ref{sec:DenoisingObservablePenalty}. Notice that \eqref{eqn:Relaxation} is a linear optimization problem, which, as discussed in Remark \ref{rem:CompMOT}, resembles an MOT problem. For this reason we expect to be able to use computational OT techniques to solve~\eqref{eq:FModel}.
\end{remark}

\begin{remark}
	We do not claim uniqueness of solutions of \eqref{eq:FModel}. This non-uniqueness may not be surprising, since problem \eqref{eq:FModel} is in general non-convex in $\delta$. 
\end{remark}

Next, we discuss the behavior of solutions to problem \eqref{eq:FModel} as the parameter $\tau\rightarrow 0$. We show in Theorem~\ref{thm:Soft-F} (see Section~\ref{pf:Theo-3.4} for its proof) that under the identifiability assumption stated below, we can recover $\delta^*$, the OT-based denoiser, from the solutions of \eqref{eq:FModel}.

\begin{assumption}\label{assump:Identifiability}
	The following identifiability condition on $\{ p(\cdot \mid \theta)  \}_{\theta \in \Omega}$ holds: If $\int_{\Omega} p(\cdot \mid \theta) \, dG(\theta) = \int_{\Omega} p(\cdot \mid \theta) \, dG'(\theta) $ for two probability measures $G$  and $G'$ over $\Omega$, then $G = G'$.
	% \begin{enumerate}
	% \item There exists a constant $C>0$ such that 
	% \[  W_2( p(\cdot | \theta_1), p(\cdot | \theta _2)) \leq C | \theta_1 - \theta_2 |, \quad \forall \theta_1, \theta_2 \in \Omega.\]
	% In other words, we assume that the family $\{ p(\cdot|\theta) \}_{\theta \in \Omega}$ is Lipschitz in the parameter $\theta$, when the family is seen as a subset of the space of probability meausres over ${\R^d}$ endowed with the $2$-Wasserstein distance.
	% \item We assume that the model is identifiable in the following sense: if $ \{ G_n \}_{n \in \mathbb{N}} $ is such that
	% \[ W_2( \mu_{G_n} , \mu  ) \rightarrow 0, \]
	% then 
	% \[ W_2(G_n , G^*) \rightarrow 0 \]
	% \end{enumerate}
\end{assumption}

\begin{theorem}\label{thm:Soft-F}
	Let $\{ \tau_n\}_{n \ge 1}$ be a sequence of positive numbers converging to $0$. Let $\delta_n^*$ be a solution to problem \eqref{eq:FModel} with $\tau=\tau_n$ (we know solutions exist thanks to Theorem~\ref{cor:StructureSolutions}). Then, under the same assumptions as in Theorem \ref{cor:StructureSolutions} and the additional Assumption \ref{assump:Identifiability},  $\delta_n^*$ converges in $\mathbb{L}^2({\R^d} : \R^m; \mu)$ to $\delta^*$ as defined in Theorem \ref{thm:Main1}. In other words, 
	\[ \lim_{n \rightarrow  \infty} \int |\delta_n^*(z)- \delta^*(z) |^2 \,d\mu(z) =0.  \]
\end{theorem}

\begin{remark}[Non-identifiable version of Theorem~\ref{thm:Soft-F}]
	\label{rem:NonIdentifiable}
	An inspection of the proof of Theorem \ref{thm:Soft-F} reveals that if we drop Assumption \ref{assump:Identifiability}, then we can conclude that the set of accumulation points of $\{  \delta_n^* \}_{n \in \N}$ in the (strong) $\mathbb{L}^2({\R^d}: \R^m; \mu)$ topology is contained in the set of minimizers of the problem
	\[\min_{\delta \: : \:   \mu_\delta= \mu  }  \E_{(Z, \Theta)\sim P_{Z,\Theta} } \left[ | \delta(Z) - \Theta|^2  \right] . \]
	In other words, from the family of problems \eqref{eq:FModel} we can find a map $\delta$ with the smallest risk attainable within the set of maps that consistently reproduce the distribution of observations $\mu$. 
	
\end{remark}

\begin{remark}
Theorem \ref{thm:Soft-F} suggests taking small values of $\tau$ in \eqref{eq:FModel} (or in its equivalent formulation \eqref{eqn:Relaxation}) to recover the OT-based denoiser. However, we anticipate a certain computational hardness for the optimization problem \eqref{eq:FModel} when $\tau$ is small. To better appreciate this, observe that small values of $\tau$ in the equivalent formulation \eqref{eqn:Relaxation} essentially enforce the hard constraint
\[ \int p(\cdot |\theta)d \gamma_2(\theta)= \mu, \]
 which is equivalent to solving a deconvolution problem.

\end{remark}
\nc

\nc 

\section{Proofs of main results from Section \ref{sec:DenoisingLatentPenalty}}
\label{sec:Proofs}

\subsection{Proof of Theorem \ref{thm:Main1}}
In order to prove Theorem \ref{thm:Main1}, we first present some preliminary results relating solutions of problem \eqref{eqn:Kantorovich} (with the cost function as in~\eqref{eq:c_z_y}) and its dual with solutions of the problem 
\begin{equation}
\min_{\pi \in \Gamma(\overline{\theta}_{\sharp} \mu, G^*)} \int \int |\theta -\vartheta|^2 d\pi(\theta, \vartheta)
\label{eq:KantorovichPushForwrad}
\end{equation}
and its dual.

\begin{prop}
	\label{lem:RelationshipKantorovic}
	Let $\tilde \pi$ and $(\tilde{\phi}, \tilde \psi )$ be solutions to \eqref{eq:KantorovichPushForwrad} and its dual, respectively. Suppose that Assumption \ref{assump:SecondMomentsG} holds. Then \eqref{eq:KantorovichPushForwrad} = \eqref{eqn:Kantorovich}. Furthermore, the functions $(\tilde \phi \circ \overline{\theta}, \tilde \psi )  $ form a solution pair for the dual of \eqref{eqn:Kantorovich}. In addition,  the coupling $\pi$ defined according to
	\begin{equation}
	d \pi(z, \theta) :=   d \tilde{\pi}( \theta \mid \overline{\theta} (z)) \,  d\mu(z) 
	\label{eqn:FromPi}
	\end{equation}
	is a solution for \eqref{eqn:Kantorovich}; here, by $\tilde{\pi}( \cdot \mid \vartheta) $ we mean the conditional distribution of $\theta$ given $\vartheta$ when $(\theta, \vartheta) \sim \widetilde{\pi}$.
\end{prop}

\begin{proof}
	Using the Kantorovich duality theorem (see Theorem 1.3 in \cite{Villani}), it follows that
	\[ \int \tilde \phi(\vartheta)d \overline{\theta}_{\sharp} \mu(\vartheta) + \int \tilde \psi(\theta)dG^*(\theta) =  \int  |\theta - \vartheta|^2 d\tilde \pi(\vartheta, \theta).  \]
	Now, the left-hand side of the above display can be written as
	\[  \int \tilde \phi(\ThetaBar(z))d \mu(z) + \int \tilde \psi(\theta)dG^*(\theta), \]
	while the right-hand side can be written, using the disintegration theorem, as
	\begin{eqnarray*}
		\int \left(\int |\theta - \vartheta|^2 d\tilde \pi(\theta|\vartheta) \right) d \ThetaBar_{\sharp}\mu(\vartheta) & = &\int \left(\int |\ThetaBar(z) - \theta|^2 d\tilde \pi(\theta|\ThetaBar(z)) \right)d\mu(z) \\
		& = & \int \int |\ThetaBar(z)- \theta|^2 d \pi(z,\theta).
	\end{eqnarray*}
	It follows that 
	\[  \int \tilde \phi\circ  \ThetaBar(z)d \mu(z) + \int \tilde \psi(\theta)dG^*(\theta)= \int |\ThetaBar(z)- \theta|^2 d\pi(z,\theta),  \]
	implying that $\pi$ and $(\tilde \phi \circ \overline{\theta}, \tilde \psi )  $ are solutions of \eqref{eqn:Kantorovich} and its dual, respectively. This computation also shows that $\eqref{eq:KantorovichPushForwrad} = \eqref{eqn:Kantorovich}$, as claimed.
\end{proof}

For the uniqueness statement in Theorem \ref{thm:Main1} we'll establish a converse statement to Proposition \ref{lem:RelationshipKantorovic}. Namely, we will prove that any solution to \eqref{eqn:Kantorovich} must have the form \eqref{eqn:FromPi}. We notice that without the additional Assumption \ref{assump:PosteriorMean} this converse statement may fail, as the next remark illustrates.

\begin{remark} 
	In general, a converse statement to Proposition~\ref{lem:RelationshipKantorovic} may not be true if Assumption~\ref{assump:PosteriorMean} does not hold (i.e., if $\overline{\theta}_{\sharp} \mu$ is not absolutely continuous w.r.t.~the Lebesgue measure), as the following example illustrates. Let $G^*$ be the uniform measure over the set $\Omega :=\{ 0,1,3,4\}$, and for every $\theta \in \Omega $, let $p(\cdot \mid \theta)$ be the uniform distribution on the interval $[0,1]$.  Then, we can see that $\mu$ is the uniform distribution on $[0,1]$ and $\overline{\theta}(\cdot) \equiv 2$, which implies that $\overline{\theta}_{\sharp} \mu = \delta_{2}$, the Dirac delta measure at the point 2. Since $\overline{\theta}_{\sharp} \mu$ is concentrated at a point, there is a unique solution $\hat{\pi}$ to problem \eqref{eq:KantorovichPushForwrad} (in fact, there is only one coupling between a Dirac delta measure and an arbitrary probability measure). However, as $\overline{\theta}(\cdot)$ is a constant, any coupling between $\mu = $Uniform[0,1] and the uniform distribution on $\Omega$ would have the same cost; hence there are actually multiple solutions to problem \eqref{eqn:Kantorovich}. %The issue in this setting is that the map $\overline{\theta}(\cdot)$ is constant, implying that $\overline{\theta}_{\sharp} \mu$ is not absolutely continuous w.r.t.~the Lebesgue measure.
\end{remark}

% In what follows we use a measure and integration argument that we combine with Brenier's theorem \cite{Brenier} for the standard squared Euclidean distance cost problem to prove the converse in Proposition \ref{lem:RelationshipKantorovic}. An alternative approach to prove the converse is to adapt the proof of Brenier's theorem for problem \eqref{eqn:Kantorovich} following the same general arguments presented in Chapter 1.3 in \cite{Santambrogio}, i.e., identifying optimality conditions for optimal couplings and using Assumptions \ref{assump:PosteriorMean} and \ref{assump:SecondMomentsG} to guarantee that \textit{any} optimal coupling must be induced by a transport map, which, in particular, would also imply uniqueness of optimal couplings (see Remark 1.19 in \cite{Santambrogio}). Through the first approach, however, we will be able to recognize the structure \eqref{eq:delta-OT} for the optimal map $\delta^*$ more directly. 

% In fact, we show below that as soon as $\overline{\theta}_{\sharp} \mu$ is absolutely continuous w.r.t.~the Lebesgue measure, then~\eqref{eqn:Kantorovich} has a unique solution. {\color{red} The following result is technical ... }

In what follows we let $\nu\in \mathcal{P}({\R^d}\times \R^m)$ be the joint distribution of $(Z,\overline{\theta}(Z))$ where $Z \sim \mu$. We use the disintegration theorem to write $\nu$ as
\begin{equation}
d \nu (z, \theta) = d \nu(z \mid \theta) \,d (\overline{\theta}_{\sharp} \mu )(\theta), \qquad \mbox{for} \;\; \theta \in \R^m, z \in {\R^d}.
\label{eqn:JointZTheta}
\end{equation}
Notice that the support of $\nu(\cdot \mid \theta)$ can be assumed to be contained in $\{ z \in {\R^d} \: : \: \overline{\theta}(z)= \theta\}$.

\begin{lemma}
	\label{lem:ForUniqueness}
	Let $\pi_0 \in \Gamma(\mu, G^*)$ and let $\hat{\pi}:=(\overline{\theta}, {\rm Id})_{\sharp} \pi_0 \; \in \Gamma(\bar{\theta}_\sharp \mu,G^*)$, where  $(\overline{\theta}, \mathrm{Id}): (z, \vartheta) \mapsto (\overline{\theta}(z), \vartheta )$. Suppose, in addition, that $\hat{\pi}$ is known to have the form $ \hat \pi = ({\rm Id} \times T)_{\sharp} (\overline{\theta}_{\sharp} \mu )$ for some map $T$. Then 
	\[ \pi_0(\cdot \mid z) = \delta_{T\circ \overline{\theta}(z)}(\cdot) \]
	for $\mu$-a.e.~$z$. In the above, $\pi_0(\cdot \mid z)$ stands for the conditional distribution of $\theta$ given $z$ when $(z, \theta) \sim \pi_0$. In particular,
	\[  \pi_0 = ({\rm Id} \times T \circ \overline{\theta})_{\sharp} \mu.    \]
\end{lemma}
\begin{proof}
	For $(\theta ,\tilde \theta ) \sim \hat{\pi}$ of the form $\hat{\pi}= ({\rm Id}, T)_{\sharp} ( \overline{\theta}_\sharp \mu) \; \in \Gamma(\bar{\theta}_\sharp \mu,G^*)$ it is clear that 
	\begin{equation}
	\hat{\pi}(\cdot \mid \theta) = \delta_{T(\theta)}
	\label{eqn:AuxhatPi}
	\end{equation}
	for $\overline{\theta}_{\sharp} \mu$-a.e. $\theta$. On the other hand, from the representation $\hat{\pi}= (\overline{\theta}, {\rm Id})_{\sharp} \pi_0$, for any bounded and measurable function $\phi: \R^m \times \R^m \to \R$ we have
	\begin{align*}
	\int_{\R^m} \int_{\R^m} \phi(\theta, \tilde \theta) \, d \hat{\pi}(\theta, \tilde \theta) &= \int_{\R^d} \int_{\R^m} \phi(\overline{\theta}(z), \tilde \theta) \, d {\pi}_0(z, \tilde \theta
	)
	\\&= \int_{\R^d} \left( \int_{{\R^m}} \phi(\overline{\theta}(z), \tilde \theta) \, d \pi_0(\tilde \theta \mid z) \right) d\mu(z)
	\\& = \int_{\R^d} \int_{\R^m} \left( \int_{{\R^m}} \phi(\theta, \tilde \theta) \, d \pi_0(\tilde \theta \mid z) \right) \, d \nu(z, \theta)
	\\& =\int_{\R^m}  \int_{\R^d} \int_{{\R^m}} \phi(\theta, \tilde \theta) \, d \pi_0(\tilde \theta \mid z)  \, d \nu( z \mid \theta ) \, d(\overline {\theta}_\sharp \mu )(\theta),
	% \\& =\int_{\R^m} \int_{\R^m} \int_{{\R^d}} \phi(\theta, \tilde \theta)   d \nu( z|\theta ) d \pi_0(\tilde \theta | z) d(\overline {\theta}_\sharp \mu )(\theta),
	\end{align*}
	where we recall that $\nu$ is the joint distribution of $(Z, \overline{\theta}(Z))$ for $Z \sim \mu$. From this computation and the uniqueness of conditional distributions in the disintegration theorem it follows that
	\[ \hat{\pi}(\cdot \mid  \theta) = \int_{\R^d} \pi_0(\cdot \mid z) \, d\nu(z \mid \theta),  \]
	for $\overline{\theta}_{\sharp}\mu$-a.e. $\theta$. That is, for any Borel measurable $A\subseteq {\R^m}$ we have
	\[  \hat{\pi}(A \mid \theta) = \int_{{\R^d}}\pi_0(A \mid z) \, d\nu(z \mid \theta).  \]
	Combining with \eqref{eqn:AuxhatPi}, it follows that for $\overline{\theta}_{\sharp}\mu$-a.e. $\theta$
	\[  \delta_{T(\theta)}(\cdot) =  \hat{\pi}(\cdot \mid \theta) = \int_{\R^d} \pi_0(\cdot \mid z) \,d\nu(z \mid \theta).  \]
	For a $\theta$ for which the above is true, we may take the singleton $A= \{  T(\theta)\}$ and conclude that 
	\[ 1=  \int_{\R^d} \pi_0(A \mid z) \,d\nu(z \mid \theta),  \]
	which implies that $\pi_0(A \mid z) = 1$ for $\nu(\cdot \mid \theta)$-a.e. $z$. That is,
	\[ \pi_0(\cdot  \mid z ) =\delta_{T(\theta)}. \]
	for $\nu(\cdot\mid\theta)$-a.e.~$z$. Finally, as discussed right after \eqref{eqn:JointZTheta}, for $z$ in the support of $\nu(\cdot \mid \theta)$ we have $\theta= \overline{\theta} (z)$. It then follows that for $\nu(\cdot \mid \theta)$-a.e. $z$ we have 
	\[ \pi_0(\cdot\mid z) = \delta_{T \circ \overline{\theta}(z)}.\]
	At this stage we can apply Fubini's theorem to conclude that 
	\[ \pi_0(\cdot \mid z) = \delta_{T \circ \overline{\theta}(z)}\]
	for $\mu$-a.e. $z\in {\R^d}$, completing in this way the proof.
\end{proof}

We are now ready to prove Theorem \ref{thm:Main1}.  

\begin{proof}[Proof of Theorem \ref{thm:Main1}]
	Under the assumption that $\overline{\theta}_{\sharp} \mu$ is absolutely continuous with respect to the Lebesgue measure, we can use Brenier's theorem (Theorem \ref{thm:Brenier}) to deduce that there exists a unique solution $\tilde{\pi}$ to \eqref{eq:KantorovichPushForwrad}, which has the form 
	\[ \tilde{\pi}= (\mathrm{Id} \times T)_{\sharp} (\overline{\theta}_{\sharp} \mu) \]
	for some measurable map $T$ of the form $T= \nabla \varphi $ for a convex function $\varphi$; existence of solutions to the dual of \eqref{eq:KantorovichPushForwrad} is guaranteed by Theorem 2.12 in \cite{Villani2003}. Further, from Brenier's theorem we also know that $T_\sharp (\overline{\theta}_{\sharp} \mu) = G^*$ and that $T$ minimizes the objective~\eqref{eq:OT_PushForward}. Proposition~\ref{lem:RelationshipKantorovic} then implies that
	\[ \pi := (\mathrm{Id} \times T\circ \overline{\theta})_\sharp  \mu \]
	is a solution of \eqref{eqn:Kantorovich}. %{\color{red} Can we say a few words as to why this is obvious from Proposition~\ref{lem:RelationshipKantorovic}.}
	
	It remains to show that the obtained solution to~\eqref{eqn:Kantorovich} is unique.  To see this, suppose that $\pi_0$ is a solution of \eqref{eqn:Kantorovich}, and let $\hat{\pi}:=(\overline{\theta}, {\rm Id})_\sharp \pi_0   $. It follows that
	\[ \int |\theta - \vartheta|^2 \,d\hat \pi(\theta,\vartheta) =  \int |\bar{\theta}(z) - \vartheta|^2 \,d \pi_0(z, \vartheta)  %=  \int c_{G^*}(z, \vartheta) \,d \pi_0(z, \vartheta) 
	=\eqref{eqn:Kantorovich} =\eqref{eq:KantorovichPushForwrad},   \]
	and thus $\hat \pi$ is a solution of \eqref{eq:KantorovichPushForwrad}; notice that the latter of the above equalities follows from Proposition \ref{lem:RelationshipKantorovic}. From this and the uniqueness of solutions to \eqref{eq:KantorovichPushForwrad}, by Assumption~\ref{assump:PosteriorMean}  as $\overline{\theta}_{\sharp} \mu$ is absolutely continuous w.r.t.~the Lebesgue measure, it follows that $\hat{\pi}= \tilde{\pi}$. Using the fact that $\hat{\pi}= (\overline{\theta} \times {\rm Id})_\sharp \pi_0 =  (\mathrm{Id} \times T)_{\sharp} (\overline{\theta}_{\sharp} \mu)  $ in Lemma \ref{lem:ForUniqueness} we can conclude that necessarily $ \pi_0 = (\mathrm{Id} \times T \circ \overline{\theta})_{\sharp} \mu, $ proving in this way the uniqueness of solutions to \eqref{eqn:Kantorovich}. 
\end{proof}

\begin{remark}
	\label{rem:GeneralG} Suppose that the Bayes estimator $\overline{\theta}$ satisfies Assumption \ref{assump:SecondMomentsG} and \ref{assump:PosteriorMean}. Then it can be easily seen that the above proof can be used to deduce that, for any $G \in \mathcal{P}(\Omega)$ with finite second moments (not necessarily equal to the prior $G^*$), the problem
	\[ \inf_{\pi \in \Gamma(\mu, G)} \int |\theta - \overline{\theta}(z)|^2 \,d\pi(z, \theta)  \]
	has a unique solution $\tilde \pi$. This unique solution takes the form
	\[ \tilde \pi = (\mathrm{Id}\times \tilde \delta)_{\sharp} \mu,  \]
	for $\tilde \delta $ the unique solution to the problem 
	\[ \inf_{\delta : \delta_{\sharp} \mu =  G } \E_{Z \sim \mu}[ | \overline{\theta}(Z) - \delta(Z)|^2  ].  \]
\end{remark}

\nc

% \begin{proof}[Proof of Theorem \ref{thm:Main1}] 
% Under Assumptions \ref{assump:PosteriorMean} and \ref{assump:SecondMomentsG} and assuming that $\mu$ has a density, it can be shown that $\overline{\theta}_{\sharp} \mu$ has a density with respect to the Lebesgue measure. By Brenier's theorem, it follows that problem \eqref{eq:KantorovichPushForwrad} has a unique solution $\pi^*$ of the form $(Id \times \nabla \varphi^*)_{\sharp} \overline{\theta}_{\sharp} \mu  $, where $\varphi^*$ is a convex function. From Lemma \ref{lem:RelationshipKantorovic} it thus follows that the unique optimal transport plan for problem \eqref{eqn:Kantorovich} takes the form:
% \[ (Id \times \nabla \varphi^* (\overline{\theta}(\cdot)) )_{\sharp} \mu. \]
% This establishes both Theorem \ref{thm:Main1} and Theorem \ref{thm:Main2}. 
%  \end{proof}

\subsection{Proof of Theorem \ref{thm:Main2}}
\label{sec:Interpolation}

In order to prove Theorem \ref{thm:Main2} we begin by relaxing \eqref{eq:EquivProblemWithPnelizatin} as follows:
\begin{equation}
\label{eq:EquivProblemWithPnelizatinRelax}
\inf_{ \pi, \tilde \pi  } \int |\overline{\theta}(z) - \tilde \theta |^2 \,d\pi(z, \tilde \theta) + \frac{1}{2 \tau} \int | \tilde \theta -  \theta|^2 \,d\tilde \pi(\tilde \theta,  \theta)   ,
\end{equation}
where the inf is taken over pairs $(\pi, \tilde \pi)$ satisfying: $\pi \in \mathcal{P}({\R^d} \times \R^m)$, $\tilde \pi \in \mathcal{P}(\R^m \times {\R^m})$, $\pi_1=\mu$, $\tilde \pi_2= G^*$, and $\pi_2= \tilde \pi_1$. We will characterize solutions to \eqref{eq:EquivProblemWithPnelizatinRelax} following the proof of a theorem in \cite{AguehCarlier}. We will then relate these solutions with problem \eqref{eq:EquivProblemWithPnelizatin} and with the characterization given in the statement of Theorem \ref{thm:Main2}.

\begin{lemma}
	\label{lem:EuqivalenceSoftBarycenter}
	Let $\tau >0 $. Then problem \eqref{eq:EquivProblemWithPnelizatinRelax} is equivalent to problem
	\begin{equation}
	\min_{ \gamma \in \Gamma(\mu, G^*)} \int B(z, \theta) d\gamma(z,\theta),      
	\label{eqn:BarycenterProblem}
	\end{equation}
	where $B(\cdot,\cdot)$ is the \textit{barycenter cost}:
	\[ B(z,\theta): = \min_{\tilde \theta \in \R^m } \left\{| \overline{\theta}(z) -\tilde \theta|^2 +\frac{1}{2\tau } |\tilde \theta- \theta |^2 \right\}= \frac{1}{(1+ 2 \tau)} |\overline{\theta}(z)- \theta|^2.   \]
\end{lemma}

\begin{proof}
	Let $\gamma^* \in \Gamma(\mu, G^*)$ be a solution to \eqref{eqn:BarycenterProblem} (note that a solution to~\eqref{eqn:BarycenterProblem} indeed exists). For a given $(z,\theta)$ in the support of $\gamma^*$ we consider 
	\begin{equation}\label{eq:T-z-theta}
	T(z, \theta) :=  \argmin_{ \tilde \theta \in \R^m  } \left\{  | \overline{\theta}(z) - \tilde \theta|^2 +\frac{1}{2\tau } |\tilde \theta- \theta|^2 \right\} = \frac{2\tau }{1 + 2 \tau}  \overline{\theta}(z) + \frac{1}{2\tau + 1}\theta.
	\end{equation}
	Let $\nu \in \mathcal{P}({\R^d} \times \Omega \times \R^m)$ be given by
	\[ \nu:= ({\rm Id}, T )_{\sharp}\gamma^*,    \]
	where $({\rm Id}, T)$ is the map $({\rm Id}, T): (z, \theta) \in {\R^d} \times \Omega \mapsto (z, \theta, T(z, \theta))$,
	and let 
	\[  \pi^* := P_{13\sharp} \nu, \quad  \tilde \pi^* := P_{32\sharp} \nu,   \]
	where $P_{13}(z,\theta, \tilde \theta) = (z, \tilde \theta )$ and $P_{32}(z,\theta, \tilde \theta) = (\tilde \theta, \theta)$. Notice that $(\pi^*, \tilde \pi^*)$ is a feasible pair for \eqref{eq:EquivProblemWithPnelizatinRelax}. 
	For this pair we have
	\begin{align}
	\begin{split}
	\eqref{eqn:BarycenterProblem} \,= \int B(z, \theta) \, d \gamma^*(z , \theta) & =   \int \left(  | \overline{\theta}(z) - T(z, \theta) |^2 + \frac{1}{2 \tau} | T(z, \theta)- \theta|^2 \right) \,d \gamma^*(z , \theta) 
	\\& =  \int  | \overline{\theta}(z) - \tilde \theta |^2 d \pi^* (z, \tilde \theta) + \frac{1}{2 \tau} \int  | \tilde \theta - \theta|^2 d \tilde \pi^*(\tilde \theta , \theta)
	\\& \geq \eqref{eq:EquivProblemWithPnelizatinRelax}.
	\label{eq:AuxBarycenter1}
	\end{split}
	\end{align}

	Let us now consider an arbitrary feasible pair $(\pi, \tilde \pi)$ for \eqref{eq:EquivProblemWithPnelizatinRelax}. From $(\pi, \tilde \pi)$ we can construct $\gamma \in \Gamma(\mu, G^*)$ by glueing $\pi$ and $\tilde \pi$ together as we describe next. Let $\nu_0= \pi_2 = \tilde \pi_1 $. Then, by the disintegration theorem applied to $\pi$ and $\tilde \pi$, we can decompose $\pi$ and $\tilde \pi$ in terms of conditionals relative to one of their marginals (in this case $\nu_0$): 
	\[  d\pi(z, \tilde \theta) = d\pi_{z| \tilde \theta} (z) d\nu_0(\tilde \theta), \qquad   d\tilde \pi( \tilde \theta, \theta ) = d \pi_{ \theta|  \tilde \theta} (\theta ) d\nu_0(\tilde \theta). \] 
	Using these decompositions, we define $\gamma \in \mathcal{P}({\R^d} \times \Omega)$ as the probability measure acting on smooth test functions $\varphi$ according to
	\[  \int_{\R^d} \int_\Omega \varphi(z, \theta) d \gamma(z, \theta) =  \int_{\R^m} \left( \int_{{\R^d}} \int_{\Omega} \varphi(z, \theta) d \pi_{z|\tilde \theta}(z) d\pi_{\theta|\tilde \theta}(\theta) \right)d \nu_0(\tilde \theta).  \]
	To intuitively explain the joint distribution $(Z,\Theta) \sim \gamma$ above, we consider a joint distribution on three variables  $(Z,\Theta, \tilde \Theta)$ defined as follows: $\tilde \Theta \sim \nu_0$, $Z$ and $\Theta$ are independent given $\tilde \Theta$, with $Z \mid \tilde \Theta \sim \pi_{z|\tilde \theta}$, and $\Theta \mid \tilde \Theta \sim \pi_{\theta|\tilde \theta}$. Thus, $\gamma$ is the joint distribution of $(Z,\Theta)$ according to the above model. It is straightforward to check that $\gamma \in \Gamma(\mu, G^*)$. Moreover, we have the following:
	\begin{align}
	\begin{split}
	& \hspace{-0.5in}\int  | \overline{\theta}(z) - \tilde \theta |^2 \,d \pi (z, \tilde \theta) + \frac{1}{2 \tau} \int  | \tilde \theta - \theta|^2 \,d \tilde \pi(\tilde \theta , \theta) \\
	& = \int \int  | \overline{\theta}(z) - \tilde \theta |^2 \,d \pi_{z|\tilde \theta}(z) \,d\nu_0(\tilde \theta) + \frac{1}{2 \tau} \int  \int | \tilde \theta - \theta|^2 \,d \tilde \pi_{\theta | \tilde \theta}( \theta) \,d\nu_0(\tilde \theta)
	\\& = \int \left[\int \int  \left(| \overline{\theta}(z) - \tilde \theta |^2 + \frac{1}{2\tau} |\tilde \theta - \theta|^2 \right) \,d \pi_{z|\tilde \theta}(z) d{\tilde \pi}_{\theta|\tilde \theta}(\theta) \right] \,d \nu_0(\tilde \theta)    \\& \geq \int \left(\int \int  B(z, \theta) d \pi_{z|\tilde \theta}(z) \,d {\tilde \pi}_{\theta|\tilde \theta}(\theta) \right) \,d \nu_0(\tilde \theta)  
	\\& = \int \int B(z, \theta) \,d\gamma(z , \theta) \;\geq\; \eqref{eqn:BarycenterProblem}. 
	\end{split}
	\label{eq:AuxBarycenter2}
	\end{align}
	Since the above is true for any arbitrary feasible pair $(\pi, \tilde \pi)$, we deduce that $\eqref{eq:EquivProblemWithPnelizatinRelax} \geq \eqref{eqn:BarycenterProblem}  $. Combining with \eqref{eq:AuxBarycenter1} we obtain the equality. 
	
	From the equality $\eqref{eq:EquivProblemWithPnelizatinRelax} = \eqref{eqn:BarycenterProblem} $, we see from \eqref{eq:AuxBarycenter1} and \eqref{eq:AuxBarycenter2} that there is an explicit way to map solutions $\gamma^*$ of \eqref{eqn:BarycenterProblem} to solutions $(\pi^*, \tilde \pi^*)$ of \eqref{eq:EquivProblemWithPnelizatinRelax} and viceversa.
\end{proof}

\begin{lemma}
	\label{lem:UniquenessBarycenter}
	For any fixed $\tau>0$ problem \eqref{eqn:BarycenterProblem} is equivalent to \eqref{eqn:Kantorovich}. In particular, its unique solution $\gamma^*$ has the form
	\[ \gamma^* = ({\rm Id} \times \delta^*)_\sharp \mu \]
	for $\delta^*$ as defined in \eqref{eq:delta-OT}. 
\end{lemma}

\begin{proof}
	Notice that a direct computation reveals, from Lemma~\ref{lem:EuqivalenceSoftBarycenter},  that for all $z\in {\R^d}$ and $\theta \in \Omega$,
	$B(z,\theta) =  (1 + 2 \tau)^{-1}  |\overline{\theta}(z)- \theta|^2$. Therefore, problem \eqref{eqn:BarycenterProblem} is equivalent to problem \eqref{eqn:Kantorovich}.
\end{proof}

\begin{lemma}
	\label{lem:UniqueSolnRelaxed}
	Problem \eqref{eq:EquivProblemWithPnelizatinRelax} has a unique solution, which is given by 
	\begin{equation}
	\pi^*=F_{\sharp} \mu, \qquad \tilde \pi ^*=\tilde{F}_{\sharp} \mu  ,
	\label{aux:TheoremSoft}
	\end{equation}
	where  $F  : z \in {\R^d} \mapsto (z, \delta_\tau^*(z)) $ and $ \tilde F : z\in {\R^d} \mapsto (  \delta_\tau^* (z) , \delta^*(z)) $. Here $\delta^*_\tau$ is defined via \eqref{eq:delta-OT-tau}.
\end{lemma}

\begin{proof}
	First, notice that from the proof of Lemma \ref{lem:EuqivalenceSoftBarycenter} we know that using $\gamma^* = ({\rm Id} \times \delta^* )_{\sharp}\mu$ we can construct a solution $(\pi, \tilde \pi)$ of \eqref{eq:EquivProblemWithPnelizatinRelax} according to 
	\[ \pi =  P_{13 \sharp} (({\rm Id}, T)_{\sharp} \gamma^*), \quad  \tilde \pi =  P_{32 \sharp} (({\rm Id}, T)_{\sharp} \gamma^*).\]
	Recall $T(\cdot,\cdot)$ from~\eqref{eq:T-z-theta} and note that $T(z, \delta(z))  = \delta^*_\tau(z)$. Using the form of $\gamma^*$, it is straightforward to verify that $\pi$ and $\tilde \pi$ defined above have the form in \eqref{aux:TheoremSoft}. It remains to show that this solution is unique.

	To see this, let $(\pi^*, \tilde{\pi}^*)$ be an arbitrary solution to \eqref{eq:EquivProblemWithPnelizatinRelax}. Let $\Upsilon$ be the probability measure over ${\R^d} \times \R^m \times \Omega$ defined by  
	\[ \int \psi(z, \tilde \theta, \theta) \, d\Upsilon(z, \tilde \theta, \theta) = \int \int \int \psi(z, \tilde \theta, \theta) \,d \pi^*_{z|\tilde \theta}(z) \,d \tilde \pi^*_{\theta| \tilde \theta}(\theta) \, d\nu_0(\tilde \theta),\]
	for all smooth test functions $\psi$; here recall the definitions of $\pi^*_{z|\tilde \theta}(\cdot)$, $ \tilde \pi^*_{\theta| \tilde \theta}(\cdot)$ and $\nu_0$ from the proof of Lemma~\ref{lem:EuqivalenceSoftBarycenter}.
	Using \eqref{eq:AuxBarycenter2} and the fact that the following inequality is actually an equality (and that both are integrals w.r.t.~the measure $\Upsilon$), we can deduce that for $\Upsilon$-a.e.~$(z,\tilde \theta, \theta)$ we have $\tilde{\theta}= T(z, \theta)$ (as the integrands must be a.e.~equal; cf.~\eqref{eq:T-z-theta}). From this it follows that the joint distribution $\Upsilon$ is determined by the joint distribution of $(z,\theta)$ and the other variable $\tilde \theta$ is a deterministic function of $(z,\theta)$, i.e.,~$\Upsilon = H_{\sharp} (P_{13 \sharp} \Upsilon ),$
	where $H: (z, \theta) \mapsto (z, T(z, \theta), \theta) $. From \eqref{eq:AuxBarycenter2} we can also deduce that $P_{13\sharp} \Upsilon$ is a solution of \eqref{eqn:BarycenterProblem}, which by Lemma \ref{lem:UniquenessBarycenter} must be equal to $({\rm Id}\times \delta^*)_{\sharp} \mu$. Therefore, 
	\[ \Upsilon = H_{\sharp}( ({\rm Id}\times \delta^*)_{\sharp} \mu  ).  \]
	From this and the fact that by construction we have $\pi^*= P_{12\sharp}  \Upsilon$ and $\tilde \pi^* = P_{23\sharp }\Upsilon $ it follows that $\pi^*$ and $\tilde \pi^*$ are as in \eqref{aux:TheoremSoft}.
\end{proof}

\begin{proof}[Proof of Theorem \ref{thm:Main2}]
	
	Let $\delta: {\R^d} \to \R^m$ be an arbitrary measurable map. Let $\tilde \pi$ be a $2$-OT transport plan between $\delta_{\sharp} \mu$ and $G^*$, and let $\pi=({\rm Id}, \delta)_{\sharp} \mu $. We see that $(\pi, \tilde \pi)$ is a feasible pair for \eqref{eq:EquivProblemWithPnelizatinRelax} and that 
	\[  \E_{Z \sim \mu}[ | \delta(Z) - \overline{\theta}(Z) | ] + \frac{1}{2\tau} W_2^2(\delta_{\sharp}\mu, G^*) = \int | \tilde \theta - \overline{\theta}(z)|^2 \,d\pi(z, \tilde \theta) + \frac{1}{2\tau } \int | \tilde \theta - \theta|^2 \,d\tilde \pi(\tilde \theta, \theta). \]
	
	From the above and the form of the unique solution to \eqref{eq:EquivProblemWithPnelizatinRelax} deduced in Lemma \ref{lem:UniqueSolnRelaxed} it follows that problem \eqref{eq:EquivProblemWithPnelizatin} admits a unique solution, which must have the form \eqref{eq:delta-OT-tau}.
\end{proof}
\nc

\nc

\section{Proofs of main results from Section \ref{sec:DenoisingObservablePenalty}}

\label{sec:DenoisingObservablePenaltyProofs}

% \subsection{Proofs of Theorems \ref{lem:StructureGammastar}, \ref{cor:StructureSolutions}, and \ref{thm:Soft-F}}
% \label{sec:ProofsRelaxation}

\subsection{Proof of Theorem \ref{lem:StructureGammastar}}\label{pf:Theo-3.2}
\begin{proof}
	First we establish the existence of solutions to \eqref{eqn:Relaxation}. Let $\{ \gamma^n \}_{n \in \mathbb{N}} \subseteq \A$ be a minimizing sequence for the objective function in \eqref{eqn:Relaxation}; we recall that $\A$, defined in \eqref{def:FeasibleSet}, is the feasible set for problem \eqref{eqn:Relaxation}. In particular, we suppose that
	\[ \lim_{n \rightarrow \infty}  \int \cc_\tau d\gamma^n =  \inf_{\gamma \in \A} \int \cc_\tau d \gamma =: M_0 <+\infty; \]
	The fact that $M_0$ is finite follows from the fact that we can take the coupling $\gamma= F_{\sharp} P_{Z,\Theta}$ (recall $P_{Z , \Theta}$ is the joint distribution of $Z$ and $\Theta$) with $F(z,\theta):=(z, \theta, z, z)$, for which one can see (by Assumption~\ref{assump:SecondMomentsG}) that $\int \cc_\tau d \gamma < +\infty$ and $\gamma \in \A$. Without the loss of generality we assume that
	\[  \sup_{ n \in \N} \int \cc_\tau d\gamma^n \leq 2 M_0. \]
	First, we prove that the sequence $\{\gamma^n \}_{n \in \N}$ is precompact in the weak sense. By Prokhorov's theorem it suffices to prove that the sequence $\{\gamma^n \}_{n \in \N}$ is tight. To see this, notice that
	\[  \int  |\theta|^2 d\gamma_{2}^n(\theta) \leq 2  \int|\theta- \overline{\theta}(z_1)|^2 d\gamma^n(z_1,\theta, z_3, z_4 ) + 2\int |\overline{\theta}(z) |^2 d\mu(z) 
	\leq 4 M_0 + 2\int |\overline{\theta}(z) |^2 d\mu(z),      \]
	which follows from the elementary pointwise inequality $ |\theta|^2 \leq 2|\theta - \overline{\theta}(z)|^2  + 2|\overline{\theta}(z)|^2  $
	and a subsequent integration with respect to $\gamma^n$ on both sides. Likewise,
	\[  \int  |z_3|^2 d\gamma_{3}^n(z_3) \leq 2 \int|z_3 - z_4|^2 d\gamma^n(z_1,\theta, z_3, z_4 ) + 2\int |z_4|^2 d\mu(z_4) \leq  8 \tau M_0 + 2\int |z_4|^2 d\mu(z_4). \]
	From the above we can conclude that all second moments of the family of distributions $\{\gamma^n \}_{n \in \N}$ are uniformly bounded, and thus the family $\{ \gamma^n \}_{n \in \N}$ is indeed tight. It follows that, up to the extraction of a subsequence that is not relabeled, $\gamma^n$ converges weakly, as $n \rightarrow \infty$, toward a limit that we will denote by $\gamma^*$.
	
	Next we show that the limiting $\gamma^*$ must be feasible for \eqref{eqn:Relaxation}, i.e., it must belong to the feasible set $\A$. First, observe that $\gamma_1^* = \gamma_4^*=\mu$  follows from the weak convergence of $\gamma^n$ toward $\gamma^*$ and the fact that for all $n \in \N$ we have $\gamma^n_1= \gamma^n_4 = \mu$. To check that $\int p(\cdot \mid \theta )d\gamma^*_2(\theta) = \gamma_3^*(\cdot) $, and thus conclude that $\gamma^* \in \A$, it is sufficient to show that
	\begin{equation}\label{eqn:AuxExistenceRElaxa}
	\int \int \varphi(z) p(z \mid \theta) \, dz \,d\gamma^*_2(\theta) =  \int \varphi(z) \, d\gamma_3^*(z)
	\end{equation}
	for all $\varphi \in C_b(\R^d)$ (here $C_b(\R^d)$ is the set of all bounded continuous functions from $\R^d$ to $\R$). To see this, first notice that
	\[  \int \int \varphi(z) p(dz|\theta) d\gamma^*_2(\theta) = \int\left( \int \varphi(z) p(dz|\theta) \right)d\gamma^*_2(\theta) = \lim_{n \rightarrow \infty} \int\left( \int \varphi(z) p(dz|\theta) \right)d\gamma^n_2(\theta), \]
	which follows from the fact that the function $\theta \in \Omega \mapsto \int \varphi(z)p(dz|\theta) $ belongs to $C_b(\Omega)$ by Assumption \ref{assump:weak-cont} and the fact that $\varphi$ is bounded and the weak convergence of $\gamma_2^n $ toward $\gamma_2^*$ as $n \rightarrow \infty$. On the other hand, the fact that $\gamma^n \in \A$ for all $n$ and the weak convergence of $\gamma_3^n $ to $\gamma^*_3$ imply that
	\[ \lim_{n \rightarrow \infty} \int\left( \int \varphi(z) p(dz|\theta) \right)d\gamma^n_2(\theta) = \lim_{n \rightarrow \infty}    \int \varphi(z) d \gamma_3^n(z) = \int \varphi(z) d\gamma_3^*(z). \]
	Combining these identities we deduce \eqref{eqn:AuxExistenceRElaxa}.

	To show that $\gamma^*$ is a solution of \eqref{eqn:Relaxation}, we start by noticing that, thanks to Assumption \ref{assump:ContinuityThetaBar}, there is a set $B\subseteq \R^d$ with $\mu(B)=1$ in which the function $\overline{\theta}(\cdot)$ is continuous. Let $\mathcal{B}:= B \times \Omega \times \R^d \times \R^d$, and notice that 
	\[ \gamma_n(\mathcal{B}) = \gamma^*(\mathcal{B}) =1 , \quad \forall n \in \N, \]
	since the first marginals of the $\gamma_n$ and $\gamma^*$ are all equal to $\mu$. We deduce that the function $\cc_\tau$ is continuous in $\mathcal{B}$. In addition, $\cc_{\tau}$ is lower bounded by a constant (because it is nonnegative). We can thus invoke Proposition 5.1.10 in \cite{ambrosio2008gradient} and from the weak convergence of $\gamma_n$ toward $\gamma^*$ deduce that
	\[ M_0 \leq  \int \cc_\tau\,  d \gamma^* \leq \liminf_{n \rightarrow \infty} \int \cc_\tau\,  d \gamma^n  = M_0. \]
	In the first inequality above we just use the fact that $M_0$ is the infimum over all couplings.
	
	% To do this, for a  $\gamma \in \A$ we can construct a suitable $\tilde{\gamma} \in \A$ with equal or smaller energy than the original $\gamma$.

	Next, we discuss the structure of solutions $\gamma^*$ of \eqref{eqn:Relaxation}. Consider an arbitrary $\gamma \in \A$ and let $\pi_{12}$ be optimal for the problem
	\[  \min_{\pi \in \Gamma(\gamma_1, \gamma_2)} \int |\theta - \overline{\theta}(z)|^2 \,d \pi(z, \theta)\]
	and let $\pi_{34}$ be optimal for the $2$-OT problem 
	\[  \min_{\pi \in \Gamma(\gamma_3, \gamma_4)} \int |z_3 - z_4|^2 \,d \pi(z_3, z_4).\] We define $\tilde{\gamma}:= \pi_{12} \otimes \pi_{34}$, i.e., $\tilde{\gamma}$ is the product measure between $\pi_{12}$ and $\pi_{34}$. Since $\tilde{\gamma}_l= \gamma_l$ for all $l=1,2,3,4$, it follows that $\tilde{\gamma} \in \A$ as well. Moreover, thanks to the fact that the cost $\cc_\tau$ is the sum of two terms without shared variables, we have
	\begin{align}
	\int \cc_\tau(z_1,\theta, z_3, z_4) \,d\tilde{\gamma}(z_1, \theta, z_3, z_4) & =  \int |\theta - \overline{\theta}(z_1)|^2 d \pi_{12}(z_1, \theta) + \frac{1}{2\tau} \int |z_3 - z_4|^2 d \pi_{34}(z_3, z_4) \notag   \\
	& \leq  \int |\theta - \overline{\theta}(z_1)|^2 d \gamma_{12}(z_1, \theta) + \frac{1}{2\tau} \int |z_3 - z_4|^2 d \gamma_{34}(z_3, z_4)    \notag\\
	& =  \int \cc_\tau(z_1,\theta, z_3, z_4) d {\gamma}(z_1, \theta, z_3, z_4). \notag
	\end{align}
	From the above display we conclude that if $\gamma=\gamma^*$ is a solution to \eqref{eqn:Relaxation}, then the inequality above must in fact be an equality, and thus we necessarily have
	\[ \int |\theta - \overline{\theta}(z_1)|^2 d \pi_{12}(z_1, \theta)   =\int |\theta - \overline{\theta}(z_1)|^2 d \gamma_{12}(z_1, \theta),\] 
	\[\int |z_3 - z_4|^2 d \pi_{34}(z_3, z_4)   =\int |z_3-z_4|^2 d \gamma_{34}(z_3, z_4).\]
	This in particular implies that $\gamma_{12}^*$ is a solution of \eqref{eq:TransportProblemCharacterizingRelaxation} and $\gamma_{34}^*$ is a $2$-OT plan between $\gamma_3^*$ and $\gamma_4^*$. The specific form $ (\mathrm{Id} \times \delta)_\sharp \mu$ for $\gamma_{12}^*$ under Assumption \ref{assump:PosteriorMean} follows from Remark \ref{rem:GeneralG}.
\end{proof}

Next we present the proof of Theorem \ref{cor:StructureSolutions}, which implies the existence of solutions to \eqref{eq:FModel} for arbitrary $\tau$.

\subsection{Proof of Theorem \ref{cor:StructureSolutions}}\label{pf:Theo-3.3}
\begin{proof}
	Let $\delta \in \mathbb{L}^2({\R^d} : \R^m; \mu)$, and consider the associated measure $\gamma_\delta$ to $\delta$ as defined in \eqref{eqn:Delta}. Then
	\begin{align}
	\begin{split}
	\eqref{eqn:Relaxation} \leq   \int \cc_\tau d \gamma_\delta &= \E_{Z \sim \mu} [  | \delta(Z) - \overline{\theta}(Z) |^2 ]  + \frac{1}{2\tau}\int |z_3 - z_4|^2 \, d\pi_{34}^\delta(z_3, z_4) 
	\\& = \E_{Z \sim \mu} [  | \delta(Z) - \overline{\theta}(Z) |^2 ]  + \frac{1}{2\tau} W_2^2(\mu_\delta, \mu). 
	\end{split}
	\label{eqn:AuxRelaxationOriginal0}
	\end{align}
	Since $\delta$ was arbitrary, the above implies that $\eqref{eqn:Relaxation} \leq \eqref{eq:FModel} - R_{\text{Bayes}} $, where we recall $R_{\text{Bayes}}$ is the Bayes risk (see the beginning of Section \ref{sec:SqError}). 
	
	Let $\gamma^*$ be a solution to problem \eqref{eqn:Relaxation}. By Theorem \ref{lem:StructureGammastar}, $\gamma_{12}^*$ can be written as $ \gamma_{12}^* = (\mathrm{Id} \times \delta_{\gamma^*})_\sharp \mu $ for some $\delta_{\gamma^*} \in \mathbb{L}^2({\R^d} : \R^m; \mu)$. In particular, $\gamma_2^* = \delta_{\gamma^* \sharp} \mu $ and thus also $\gamma_3 = \mu_{\delta_{\gamma^*}}$. From the proof of Theorem \ref{lem:StructureGammastar} we further deduce that
	\begin{align}
	\begin{split}
	\eqref{eqn:Relaxation}  & = \int \cc_\tau d \gamma^* = \int |\theta - \overline{\theta}(z_1)|^2 d \gamma_{12}^*(z_1, \theta) + \frac{1}{2\tau}  \int |z_3 - z_4|^2 d \gamma_{34}^*(z_3, z_4) 
	\\& =  \E_{Z \sim \mu} [ |\delta_{\gamma^*}(Z)  - \overline{\theta}(Z) |^2] + \frac{1}{2\tau} W_2^2(\mu_{\delta_{\gamma^*}}, \mu) 
	\\ & \geq \eqref{eq:FModel}-  R_{\text{Bayes}}.
	\end{split}
	\label{eqn:AuxRelaxationOriginal}
	\end{align}
	Combing the above two inequalities we deduce that $ \eqref{eqn:Relaxation} =  \eqref{eq:FModel}-  R_{\text{Bayes}}$ and that in \eqref{eqn:AuxRelaxationOriginal} the inequality is actually an equality. In particular, $\delta_{\gamma^*}$ is a solution to \eqref{eq:FModel}. 
	
	Conversely, now that we know that $\eqref{eqn:Relaxation} = \eqref{eq:FModel}-  R_{\text{Bayes}}$,~\eqref{eqn:AuxRelaxationOriginal0} implies that if $\delta$ is optimal for \eqref{eq:FModel}, then $\gamma_{\delta}$, as defined in \eqref{eqn:Delta}, is optimal for \eqref{eqn:Relaxation}.
\end{proof}

\subsection{Proof of Theorem~\ref{thm:Soft-F}}\label{pf:Theo-3.4}
\begin{proof}
	Let $\{ \tau_n \}_{n \in \mathbb{N}}$ be a sequence of positive numbers converging to $0$. Let $\delta_n^*$ be a solution of problem \eqref{eq:FModel} for $\tau=\tau_n$. Thanks to the second part of Theorem \ref{cor:StructureSolutions}, the measure $\gamma_{\delta_n^*}$ associated to $\delta_n^*$ that was defined in \eqref{eqn:Delta} is a solution for the problem \eqref{eqn:Relaxation} with $\tau=\tau_n$. In what follows, we use $\gamma^n$ to denote $\gamma_{\delta_n^*}$ in order to make the notation less cumbersome.
	
	Using similar arguments to those in the first part of the proof of Theorem \ref{lem:StructureGammastar}, we can show that $\{ \gamma^n\}_{n \in \N} $ is precompact in the weak topology of probability measures and that all its possible accumulation points are in $\A$. Let us then take a subsequence of $\{\gamma^n \}_{n \in \N}$ that converges weakly to some ${\gamma} \in \A$. For notational simplicity let us denote the subsequence also by $\{\gamma^n \}_{n \in \N}$. We will characterize $\gamma_{12}$, the projection of $\gamma$ onto the first two coordinates.
	
	First, observe that, since $\tau_n \rightarrow 0$, as $\int \mathbf{c}_{{\tau}_n } d \gamma^n $ is bounded from above (see the initials steps in the proof of Theorem~\ref{pf:Theo-3.2}),
	\[W_2^2(\gamma_3^n, \mu) \leq \int |z_3 - z_4|^2 \, d\gamma^n(z_3,z_4) \leq 2\tau_n \int \mathbf{c}_{{\tau}_n } d \gamma^n \rightarrow 0.\]
	In particular, $\gamma_3$, which is the limit of $\gamma_3^n$, must be equal to $\mu$. By Assumption \ref{assump:Identifiability}, we deduce from $\int p(\cdot| \theta) \, d \gamma_2(\theta) = \gamma_3(\cdot)  = \mu(\cdot) =  \int p(\cdot| \theta) \, d G^*(\theta)  $ that $\gamma_2 = G^*$. Therefore, $\gamma_{12} \in \Gamma(\mu, G^*)$. On the other hand, by weak convergence of $\gamma^n$ to $\gamma$ (and Assumption~\ref{assump:ContinuityThetaBar}) we get
	\begin{align}
	\begin{split}
	\int |\theta - \overline{\theta}(z)|^2 \, d \gamma_{12}(z, \theta) & \leq \liminf_{n \rightarrow \infty}  \int |\theta - \overline{\theta}(z)|^2 \, d \gamma_{12}^n(z, \theta)   
	\\& \leq \liminf_{n \rightarrow \infty}  \int \cc_{\tau_n} d \gamma^n.
	\end{split} 
	\label{aux:gamm12}
	\end{align}
	Now, an arbitrary $\pi_{12} \in \Gamma(\mu, G^*)$ induces a $\tilde{\gamma} \in \A$ as follows
	\[\tilde{\gamma} := \pi_{12} \otimes ( \mathrm{Id} \times \mathrm{Id}  )_{\sharp }\mu, \]
	and as can be easily verified we have
	\[ \int \cc_{\tau_n} d\tilde \gamma = \int |\theta - \overline{\theta}(z)|^2 \, d\pi_{12}(z,\theta).\]
	Since $\gamma^n$ is optimal for \eqref{eqn:Relaxation} with $\tau = \tau_n$, it follows that  
	\[ \int \cc_{\tau_n} \, d\gamma^n \leq  \int \cc_{\tau_n} d\tilde \gamma = \int |\theta - \overline{\theta}(z)|^2 d\pi_{12}(z,\theta).   \]
	Taking $\liminf$ on both sides of the above inequality, combining with \eqref{aux:gamm12}, and using the fact that $\pi_{12}$ was arbitrary, we deduce that $\gamma_{12}$ is a solution to \eqref{eqn:Kantorovich}. Theorem \ref{thm:Main1} thus implies that $\gamma_{12} = \pi^*= (\mathrm{Id} \times \delta^*)_{\sharp} \mu$, where $\delta^*$ is our OT-based denoising estimand. We have thus shown that any convergent subsequence of the original $\{\gamma^n \}_{n \in \N}$ converges to the same limit point $\pi^*:=(\mathrm{Id} \times \delta^*)_{\sharp} \mu$, and as a consequence the original sequence also converges to this same limit point.
	
	% From the above we deduce that the full sequence $\gamma_{12}^n = (\mathrm{Id} \times \delta_n^*)_{\sharp} \mu$ converges weakly to $\pi^*:=(\mathrm{Id} \times \delta^*)_{\sharp} \mu$ as $n \rightarrow \infty$.

	% When we combine with \eqref{} and use the fact that $\pi_{12} \in \Gamma(\mu, G^*)$ was arbitrary, we deduce that $\gamma_{12}$ is optimal for \eqref{} and thus it has to be equal to $(Id \times \delta^*)_{\sharp}\mu$. 

	At this stage we may use a series of results from functional analysis and measure theory that we present in Appendix \ref{app:MeasureFunctional} to deduce that $\delta_n^* \rightarrow_{\mathbb{L}^2({\R^d} : \R^m; \mu)} \delta^*$. Indeed, first notice that Lemma \ref{lem:CharacterizationTL^0} implies that $\delta_n^* $ converges to $\delta^*$ in $\mu$-measure (see definition in the statement of Lemma \ref{lem:CharacterizationTL^0}). In addition, since we also have 
	\[ \sup_{n \in \N} \int |\delta_n^*(z)|^2 d\mu(z) < \infty,  \]
	as can be easily verified, we can invoke Lemma \ref{lem:MeasureWeakL2} to conclude that $\delta_n^*$ converges weakly in $\mathbb{L}^2({\R^d} : \R^m; \mu)$ to $\delta^*$ (see Definition \ref{def:WeakL2}). In particular, we have
	\[  \lim_{n \rightarrow \infty} \int \delta_n^*(z)\cdot \overline{\theta}(z) \,d\mu(z) =   \int \delta^*(z)\cdot \overline{\theta}(z) \,d\mu(z).    \]
	Since in addition we have
	\[ \lim_{n \rightarrow \infty}   \int |\delta_n^*(z)- \overline{\theta}(z)|^2 \,d\mu(z)  = \int |\delta^*(z)- \overline{\theta}(z)|^2 \,d\mu(z),  \]
	after expanding the square we conclude that
	\[ \lim_{n \rightarrow \infty}   \int |\delta_n^*(z)|^2 d\mu(z)  = \int |\delta^*(z)|^2 d\mu(z).  \]
	Lemma~\ref{lem:weak+norm=strong} now implies that $\delta_n^*$ converges in $\mathbb{L}^2({\R^d} : \R^m; \mu)$ to $\delta^*$, as we wanted to prove.
\end{proof}

\section{Discussion and future work}
\label{sec:Discussion}

In this paper we have presented a new perspective on the denoising problem --- where one observes $Z$ (from model~\eqref{eq:Mix-Mdl}) and the goal is to predict the underlying latent variable $\Theta \sim G^*$ --- based on OT theory. We define the OT-based denoiser $\delta^*(Z)$ as the function which minimizers the Bayes risk in this problem subject to the distributional stability constraint $\delta^*(Z) \sim G^*$. Moreover, we have developed two approaches to  characterize this OT-based denoiser $\delta^*(Z)$, one where we explicitly use $G^*$ (Section~\ref{sec:DenoisingLatentPenalty}) and one where we directly involve $\mu$ (the marginal distribution of $Z$) and the likelihood model $\{p(\cdot \mid \theta)\}_{\theta \in \Omega}$ without an explicit use of the prior $G^*$ (Section~\ref{sec:DenoisingObservablePenalty}). 

One important direction that we believe is worth investigating in future work is the numerical implementation of our proposals in the finite data setting. In Appendix~\ref{sec:EmpiricalBayes} we outline an approach to implementing the sample version of the Kantorovich relaxation problem~\eqref{eqn:Kantorovich} (by directly plugging in an estimator of $G^*$) which would lead  to an estimator of $\delta^*$ (cf.~\eqref{eq:delta-OT}). We conjecture that this approach would yield a consistent estimator of $\delta^*$ and it would be interesting to study its rate of convergence. 

The adaptation of our approach in Section~\ref{sec:DenoisingObservablePenalty} to the finite data setting to find a solution to~\eqref{eq:FModel} can, in principle, avoid direct estimation of $G^*$. Here the key challenge is to find a suitable sample version of the Kantorovich relaxation problem~\eqref{eqn:Relaxation} (which under appropriate conditions yields a solution to~\eqref{eq:FModel}; see Theorem~\ref{cor:StructureSolutions}). Indeed, in contrast to the gradient descent approach outlined in \eqref{rem:GradDescent} for solving \eqref{eq:FModel} in the finite data setting, the Kantorovich relaxation \eqref{eqn:Relaxation} is a linear program whose optimizers are guaranteed to induce global solutions to \eqref{eq:FModel}. However, the first hurdle in developing an empirical version of~\eqref{eqn:Relaxation} is to  estimate the cost function $\cc_\tau$ in~\eqref{eqn:CostRelaxation} which involves the Bayes estimator $\overline{\theta} (\cdot)$. This is where Tweedie's formula (see~\eqref{eq:Tweedie} in Appendix~\ref{sec:Tweedie}) can be very useful. It expresses the posterior mean $\overline{\theta}(\cdot)$ in an exponential family model (see Appendix~\ref{sec:Exp-Family}) in terms of the marginal density $f_{G^*}$ of the observations (and its gradient) that can be estimated (nonparametrically) directly from the sample $Z_1,\ldots, Z_n$, say via kernel density estimation. Thus, Tweedie's formula can yield an estimated cost function without directly estimating the unknown prior $G^*$. The next step would be to solve problem~\eqref{eqn:Relaxation} with this estimated cost. As problem~\eqref{eqn:Relaxation} is closely reminiscent of a multimarginal OT problem we expect that some adaptations of existing computational OT tools can be useful in solving it. We leave a thorough study of this approach as future work.

\section*{Acknowledgments}
The authors are thankful to Young-Heon Kim and Brendan Pass for enlightening discussions on topics related to the content of this paper. NGT was partially supported by the NSF-DMS grant 2236447 and would also like to thank the IFDS at UW-Madison and NSF through TRIPODS grant 2023239 for their support. BS was supported by the NSF-DMS grant 2311062.

\bibliographystyle{abbrv}
\bibliography{References}

\appendix

%\subsection{Some other examples of model~\eqref{eq:Mix-Mdl}}\label{Appendix:Examples}

%\begin{example}[Uniform scale mixture]\label{ex:Unif-scale-Mix}
%Suppose that $G^*$ is a distribution on $(0,\infty)$ and $p(\cdot \mid \theta)$ corresponds to the uniform density on the interval $[0,\theta]$ (for $\theta >0$). Thus, the marginal density of $Z$ is given by $f_{G^*}(z) := \int \frac{1}{\theta} \mathbb{I}_{[0,\theta]}(z) \, dG^*(\theta) = \int_z^\infty \frac{1}{\theta} \, dG^*(\theta),$ for $z >0$. It is well-known that any (upper semicontinuous) nonincreasing density on $(0,\infty)$ can be represented as $f_{G^*}$ for a suitable $G^*$ \cite[p.~158]{Feller-71}. This class of distributions arise naturally via connections with renewal theory (see e.g.,~\cite{Woodroofe-Sun-1993}), multiple testing (see e.g.,~\cite{Langaas-Et-Al-2005, Ignatiadis-Huber-21}), etc.
%\end{example}

%\begin{example}[Poisson mixture]\label{ex:Poisson-Mix}
%Suppose that $Z$ given $\Theta = \theta$ has a Poisson distribution with mean $\theta$, i.e., $p(z \mid \theta) = e^{-\theta} \frac{\theta^z}{z!}$ where $z = 0,1,\ldots$. It is natural to assume here that $G^*$ is some distribution supported on $[0,\infty)$. This model has a long history in statistics with applications in actuarial sciences (see e.g.,~\cite{Bulhmann-Gisler-2005}), ecology~\cite[Section 6.2]{Efron-Hastie-2021}, etc.; in fact~\cite{Robbins-1956} introduced the idea of Empirical Bayes through this example.
%\end{example}

\section{Exponential families}\label{sec:Exp-Family}
Consider a random vector $Z \in \R^d$ having a density w.r.t.~a dominating measure $\lambda$, parametrized by $\theta := (\theta_1,\ldots, \theta_m) \in \R^m$ and expressible as:
\begin{equation}\label{eq:ExpFamily}
p(z\mid \theta) := \exp \Big[ \sum_{j=1}^m \theta_j T_j(z) - A(\theta) \Big] h(z), \qquad \mbox{for } \; z \in \R^p.
\end{equation}
Here $h:\R^d \to \R$ is a nonnegative function, $T = (T_1,\ldots,T_m)$ is a measurable function from $\R^d$ to $\R^m$, and the parameter space is the set 
\begin{equation}\label{eq:ParaSpace}
\Omega := \{\theta \in \R^m: A(\theta) < \infty\},
\end{equation}
where the function $A: \Omega \to \R$ (sometimes referred to as the {\it cumulant function} or the {\it log-partition function}) is defined as
\begin{equation}\label{eq:A-CGF}
A(\theta) := \log \int \exp \left[ \sum_{i=1}^m \theta_i T_i(z)\right] h(z) \, d \lambda(z).
\end{equation}
Through the discussion in this Appendix we will assume that $\Omega$ is a nonempty open subset of $\R^m$ for simplicity. 

In this case, $Z$ is said to belong to a {\it regular $m$-parameter exponential family}, and $\theta$ is the natural or canonical parametrization.  %Let $\mathcal{X} \subset \R^p$ denote the support of $X$.  	
There are many examples of parametric families belonging to an exponential family, e.g., Gaussian, binomial, multinomial, Poisson, gamma, and beta distributions, as well as many others. Here are some examples.

\begin{example}[Exponential distribution] Consider the exponential distribution parametrized by $\beta \in (0,\infty)$:
	\begin{equation}\label{eq:Exp-Dist}
	p_\beta(z) = \beta e^{-\beta z}\mathbf{1}_{(0,\infty)}(z).
	\end{equation}
	The above family is indeed a 1-parameter exponential family with natural parameter $\theta := -\beta$ and $\Omega = (-\infty, 0)$. Here $T(x) = x$, $h(x) = \mathbf{1}_{(0,\infty)}(x)$ and $A(\theta) = \log \int_0^\infty e^{\theta x} dx = \log(-\theta^{-1})$.
\end{example}

\begin{example}[Multivariate normal]\label{ex:Mult-Normal} Consider the family of multivariate normal distributions on $\R^d$ with a fixed known nonsingular covariance matrix $\Sigma \in \R^{d \times d}$ and unknown mean vector $\beta = (\beta_1,\ldots, \beta_d) \in \R^d$, i.e., $Z \sim N_d(\beta, \Sigma)$ has density given by
	\begin{equation}\label{eq:Multi-Normal}
	p_\beta(z) = \frac{e^{-\frac{1}{2} (z - \beta)^\top \Sigma^{-1} (z - \beta)}}{\sqrt{(2 \pi)^{d} |\Sigma|}}, \qquad \mbox{for } \; z \in \R^d.
	\end{equation}
	It is easy to check that~\eqref{eq:Multi-Normal} can be expressed in the form~\eqref{eq:ExpFamily} where we take $$\theta := \Sigma^{-1} \beta, \qquad T(z) := z, \qquad A(\theta) := \frac{1}{2}\theta^\top \Sigma \theta \quad \mbox{and} \quad  h(z) \equiv p_0(z) = \frac{e^{-\frac{1}{2} z^\top \Sigma^{-1} z}}{\sqrt{(2 \pi)^{d} |\Sigma|}}.$$
\end{example}
Suppose that $Z \sim p(\cdot \mid \theta)$ as in~\eqref{eq:ExpFamily}. Here are some important properties of exponential families.
\begin{enumerate}
	\item The support of $Z$ (i.e., $z$ such that $p(z \mid \theta) >0$) does not depend on $\theta$.
	
	\item It is clear that the statistic $T(Z)$ is a {\it sufficient statistic} for this family. It can be shown that\footnote{A proof of this can be obtained as follows. Recall~\eqref{eq:A-CGF}. Thus, $$e^{A(\theta)} = \int e^{\theta^\top T(z)} h(z) \, d \lambda (z) .$$ Differentiating this expression with respect to $\theta_j$, which can be done under the integral if $\theta \in \Omega^o$ (here $\Omega^o$ is the interior of $\Omega$), gives $$e^{A(\theta)} \frac{\partial A(\theta)}{\partial \theta_j} = \int T_j(z) e^{\theta^\top T(z)} h(z) \,d \lambda (z)  \quad \Rightarrow \frac{\partial A(\theta)}{\partial \theta_j} = \int T_j(z) p(z \mid \theta) \, d\lambda(z) = \E_\theta [T_j(Z)].$$ }  
	\begin{equation}\label{eq:Mean-Exp}
	\E_\theta [T_j(Z)] = \frac{\partial A(\theta)}{\partial \theta_j}, \qquad \mbox{for } \; j = 1,\ldots,m.
	\end{equation}
	
	\item The natural parameter space $\Omega$ is a {\it convex  set} and the cumulant function $A(\cdot)$ is a {\it convex function}. 
	
	\item The moment generating function of $T \equiv (T_1(Z), \ldots, T_m(Z))$ is, for $u \in \R^m$ such that $u + \theta \in \Omega$, 
	\begin{eqnarray*}
		M_T(u) := \E[e^{u^\top T}] & = & \int e^{u^\top T} e^{\theta^\top T - A(\theta)} h(z) \, d\lambda(z) \\ 
		& = & e^{A(u + \theta) - A(\theta)} \int p(z \mid {u +\theta}) d\lambda(z) = e^{A(u + \theta) - A(\theta)}.     
	\end{eqnarray*}
	
	\item The {\it cumulant generating function} is
	\begin{equation}\label{eq:CGF}
	K_T(u) := \log M_T(u) = A(u + \theta) - A(\theta).
	\end{equation} 
	
	\item Noting that if $M_T(\cdot)$ is finite in some neighborhood of the origin, then $M_T$ has continuous derivatives of all orders at the origin, and for $r_j \ge 0$, for $j = 1,\ldots, m$, $$\E [T_1^{r_1}(Z) \times \cdots \times T_s^{r_s}(Z)] = \frac{\partial^{r_1}}{\partial u_1^{r_1}} \ldots \frac{\partial^{r_s}}{\partial u_s^{r_s}} M_T(u) \Big|_{u=0}.$$ %The corresponding derivatives of $K_T$ are called cumulants, denoted $$\alpha_{r_1,\ldots, r_s} := \E [T_1^{r_1}(X) \times \cdots \times T_s^{r_s}(X)] = \frac{\partial^{r_1}}{\partial u_1^{r_1}} \ldots \frac{\partial^{r_s}}{\partial u_s^{r_s}} K_T(u) \Big|_{u=0}.$$ 
	Thus, when $r_j = 1$  and $r_k = 0$ for all $k \ne j$, we obtain~\eqref{eq:Mean-Exp}.

\end{enumerate}
See~\cite[Chapter 10]{Keener2010} and~\cite{Efron2022} for a more detailed study of exponential families.

\section{Tweedie's formula}\label{sec:Tweedie}
Now suppose that $\Theta$ is assumed to have a prior distribution $G^*$ (on $\Omega \subset \R^m$). Thus our model becomes:
\begin{equation}\label{eq:Mdl}
\Theta \sim G^* \qquad \mbox{ and } \qquad Z \mid \Theta = \theta \sim p(\cdot \mid \theta),
\end{equation}
where we assume that $p(\cdot \mid \theta)$ comes from the exponential family~\eqref{eq:ExpFamily}. 
Then the marginal density of $Z$ (w.r.t.~$\lambda$) is %\begin{equation}\label{eq:Marginal-Dens}
$$f_{G^*}(z) := \int  p(z \mid \theta) \, dG^*(\theta), \qquad \mbox{for } \; z \in \R^d.$$
%\end{equation}
Let $\mathcal{Z} \subset \R^d$ be the support of the marginal distribution of $Z$. Now Bayes rule provides the posterior density of $\Theta$ given $Z$. Suppose that $\Theta$ has density $g(\cdot)$, w.r.t.~a dominating measure $\xi$, with support contained in the set $\Omega \subset \R^m$. Then, the posterior density of $\Theta$ given $Z=z$ (w.r.t.~$\xi$) is given by,  for $\theta \in \Omega$ and $z \in \mathcal{Z}$,
\begin{equation}\label{eq:Bayes-Rule}
p_{\Theta| Z}(\theta \mid z) = \frac{p(z\mid \theta) g(\theta)}{f_{G^*}(z)} = \frac{e^{\theta^\top T(z) -A(\theta)} h(z)g(\theta)}{f_{G^*}(z)} = e^{\theta^\top T(z) -\kappa(z)} e^{-A(\theta)} g(\theta),
\end{equation}
where 
\begin{equation}\label{eq:kappa}
\kappa(z) := \log \left( \frac{f_{G^*}(z)}{h(z)} \right), \qquad \mbox{for } \;\; z \in \mathcal{Z}.\end{equation} This implies that $\Theta \mid Z = z$ is also an {\it exponential family} with canonical parameter $T(z)$, sufficient statistic $\Theta$, and log-partition function $\kappa(z)$. Thus, the cumulant generating function is (cf.~\eqref{eq:CGF})
\begin{equation}\label{eq:Post-CGF}
\log \E[e^{\Theta^\top t} \mid Z = z] = \kappa(t + z) - \kappa(z)
\end{equation}
for $z \in \mathcal{Z}$ such that $t + z \in \mathcal{Z}$.

%Next, we assume that $T(x) = x$ which is the case for the normal distribution described in Example~\ref{ex:Mult-Normal}. Thus $s = p$. 

%In this special case, letting
%\begin{equation}\label{eq:kappa}
%\kappa(x) := \log \left(\frac{f(x)}{f_0(x)} \right), \qquad \mbox{for }x \in \mathcal{X},
%\end{equation}
%we see that~\eqref{eq:Bayes-Rule} simplifies as 
%\begin{equation}\label{eq:Cond-Exp_Fam}
%f_{\Theta|X}(\theta|x) = e^{x^\top \theta - \kappa(x)} [g(\theta) e^{-A(\theta)}].
%\end{equation} 
%Thus, $\Theta|X = x$ is also an $s$-parameter exponential distribution which has the natural parametrization given in~\eqref{eq:Cond-Exp_Fam}, with parameter space $\mathcal{X}$, and support $\Omega$. 

Tweedie's formula, given below, calculates the posterior expectation of $\Theta$ given $Z = z$ in the setting~\eqref{eq:Mdl}.

\begin{lemma}[Tweedie's formula] For $z \in \mathcal{Z}$, we have 
	\begin{equation}\label{eq:Tweedie}
	\E[\Theta \mid Z = z] = \nabla \kappa(z) = \frac{\nabla f_{G^*}(z)}{f_{G^*}(z)} - \frac{\nabla h(z)}{h(z)}.
	\end{equation}
\end{lemma}
\begin{proof}
	The result is a direct consequence of the fact that the distribution of $\Theta \mid  Z = z$ is an $m$-parameter exponential family with log-partition function $\kappa(\cdot)$ defined via~\eqref{eq:kappa}: By property 2.~above (see~\eqref{eq:Mean-Exp}) the expectation of the sufficient statistic $\Theta$ can then be expressed as the gradient of the log-partition function.
	%By differentiating the cumulant generating function (as in~\eqref{eq:CGF}; also see~\eqref{eq:Mean-Exp}) we obtain. 
\end{proof}
For $d=m = 1$, the above formula for the Gaussian case was given in~\cite{Robbins-1956}.~\cite{Efron-2011} calls this Tweedie's formula since Robbins attributes it to M.C.K.~Tweedie;
however it appears earlier in~\cite{Dyson-1926} who credits it to the English astronomer Arthur
Eddington.

\begin{lemma}\label{lem:Grad-Cvx}
	Consider model~\eqref{eq:Mdl} where we assume that $p(\cdot \mid \theta)$, for $\theta \in \Omega$, is a member of an exponential family of distributions as in~\eqref{eq:ExpFamily} with $T(z) = z$ and $m=d$. Suppose further that $h(\cdot)$ in~\eqref{eq:ExpFamily} integrates to 1 (w.r.t.~$\lambda$). Then $\kappa(\cdot)$, as defined in~\eqref{eq:kappa}, is a convex function. As a consequence, $\E[\Theta \mid Z = z]$ is the gradient of a convex function.
\end{lemma}
\begin{proof}
	Observe that under the assumptions of the lemma, from~\eqref{eq:kappa} we see that the distribution of $\Theta \mid Z = z$ is an $m$-parameter exponential family with log-partition function $\kappa(\cdot)$ defined by~\eqref{eq:kappa}. As the log-partition function $\kappa(\cdot)$ is known to be convex, the result follows. 
\end{proof}

\begin{remark}[Tweedie's formula for multivariate normal distribution]
	Suppose now that $Z$ has multivariate normal distribution with known covariance matrix as in Example~\ref{ex:Mult-Normal}. Then, for $z \in \R^d$, $$\E[\Theta \mid Z = z] = \nabla \kappa(z) = \Sigma^{-1} z + \frac{\nabla f_{G^*}(z)}{f_{G^*}(z)},$$ where the last equality follows from~\eqref{eq:Tweedie} and the fact that $\nabla h(z) = - h(z)(\Sigma^{-1} z) $. Thus, the Bayes estimator of mean $\mu$ in~\eqref{eq:Multi-Normal} is 
	\begin{equation}\label{eq:Tweedie-2}
	\E[\mu \mid Z = z] = z + \Sigma \frac{\nabla f_{G^*}(z)}{f_{G^*}(z)}.
	\end{equation}
\end{remark}

\section{Auxiliary results from measure theory and functional analysis}
\label{app:MeasureFunctional}

\begin{lemma}
	\label{lem:CharacterizationTL^0}
	Let $\mu$ be a Borel probability measure over $\R^d$. Suppose that $\{ \mathrm{T}_n \}_{n \in \N}$ is a sequence of (vector valued) Borel measurable maps $\mathrm{T}_n: \R^d \rightarrow \R^m$ and suppose that $\mathrm{T}$ is another Borel measurable map from $\R^d$ into $\R^m$.  
	
	The sequence of measures $\pi_n =: ( \mathrm{Id} \times \mathrm{T}_n)_{\sharp} \mu $ converges weakly to $\pi:= (\mathrm{Id} \times \mathrm{T})_{\sharp } \mu$ if and only if $\mathrm{T}_n $ converges in $\mu$-measure to $\mathrm{T}$, i.e., for every $\eta>0$ we have
	\[ \lim_{n \rightarrow \infty} \mu \left( \left\{ z \in \R^d \: : \: |\mathrm{T}_n(z) -\mathrm{T}(z) | \geq \eta  \right\} \right)=0. \]
\end{lemma}

\begin{proof}
	Recall that weak convergence of probability measures is equivalent to convergence in Levy-Prokhorov metric, which we recall is defined as:
	\[ d_{LP}(\pi_n, \pi) := \inf \left\{\varepsilon>0 \: : \: \pi_n(A) \leq  \pi(A^\veps) + \veps \text{ and }  \pi(A) \leq  \pi_n(A^\varepsilon) + \veps, \quad \forall A \in \mathfrak{B}(\R^d \times \R^m)     \right\}.  \]
	In the above, for an arbitrary $A$ the set $A^\veps$ is defined as the set of points $(z, \theta)$ such that there exists $(\tilde z , \tilde \theta) \in A$ with $|z-\tilde z | + |\theta -\tilde \theta| <\veps $. 
	% We use  $\mathfrak{B}(\R^d \times \R^m)$ to denote the Borel $\sigma$-algebra over $\R^d \times \R^m$. 

	Let us first assume that $\pi_n$ converges weakly to $\pi$ and let $\veps_n:= 2d_{LP}(\pi_n, \pi)$. Fix $\eta>0$ and $r>0$. From the fact that $\mu$ is a Borel probability measure over $\R^d$, it follows that $\mathrm{T}$ can be approximated in the $\mu$-a.e. convergence sense by a sequence of Lipschitz continuous functions (with possibly growing Lipschitz constants). Indeed, by density (in the $\mu$-a.e. sense) of simple functions in the set of all measurable functions and the fact that we are considering the Borel $\sigma$-algebra (which is generated by open sets) one can reduce the problem to approximating (scalar) indicator functions of open sets. In turn, using rescaled distance functions (which are Lipschitz), one can easily  approximate indicator functions of open sets with Lipschitz continuous functions as desired.  It thus follows that there exists a Lipschitz function $\psi_r: \R^d \rightarrow \R^m$ such that 
	\[  \mu(G_r) \leq r  \]
	for the set $G_r$ defined as
	\[ G_r:= \left\{ z \in \R^d \,:\,|\psi_r(z) - \mathrm{T}(z)|\geq \frac{\eta}{3}  \right\}.  \]
	The above says that we can approximate the Borel measurable function $\mathrm{T}$ up to accuracy $\eta/3$ by the Lipschitz function $\psi_r$ on a set with ``large'' $\mu$-probability. 
	Intersecting the set $\left\{ z \in \R^d \: : \: |\mathrm{T}_n(z) -\mathrm{T}(z) | \geq \eta  \right\}$ with $G_r$ and, separately, with $G_r^c$, we get the inequality
	\[  \mu \left( \left\{ z \in \R^d \: : \: |\mathrm{T}_n(z) -\mathrm{T}(z) | \geq \eta  \right\} \right) \leq r + \mu \left( \left\{ z \in \R^d \: : \: |\mathrm{T}_n(z) -\psi_r(z) | \geq \frac{2}{3}\eta  \right\} \right).      \]
	Let us now consider the set 
	\[ A_r:= \left\{ (z, \theta) \in \R^d \times \R^m\: : \: |\psi_r(z) -\theta|\geq \frac{2}{3}\eta \right\}. \]
	Thanks to the specific form of the measure $\pi_n$, we can write
	\begin{align*}
	\mu\left( \left\{  z \in \R^d \,:\, |\psi_r(z) - \mathrm{T}_n(z)|\geq \frac{2}{3}\eta \right\} \right) & = \pi_n(A_r) \leq \pi(A_r^{\veps_n}) + \veps_n.
	\end{align*}
	On the other hand, 
	\begin{align*}
	\pi(A_r^{\veps_n}) & =   \mu\left( \left\{ z  \in  \R^d \, : \,  \exists\, (\tilde z , \tilde \theta) \text{ s.t. } |z-\tilde z | +| \tilde \theta - \mathrm{T}(z) | < \veps_n \text{ and } |\tilde \theta - \psi_r(\tilde z)| \geq \frac{2}{3}\eta \right\} \right) 
	\\& \leq \mu\left( \left\{ z  \in  \R^d \, : \,  \exists\, \tilde z \text{ s.t. } |z-\tilde z | < \veps_n \text{ and } |\mathrm{T}(z) - \psi_r(\tilde z)| \geq \frac{2}{3}\eta -\veps_n \right\} \right).
	\end{align*}
	In turn, we see that 
	\[ \mu \left( G_r \cap  \left\{ z \in \R^d \,: \, \exists\, \tilde z \text{ s.t. } |z-\tilde z | < \veps_n \text{ and } |\mathrm{T}(z) - \psi_r(\tilde z)| \geq \frac{2}{3}\eta -\veps_n \right\}   \right) \leq r, \]
	and $\mu( G_r^c \cap  \{ z \in \R^d \,: \,  \exists \, \tilde z \text{ s.t. } |z-\tilde z | < \veps_n \text{ and } |\mathrm{T}(z) - \psi_r(\tilde z)| \geq \frac{2}{3}\eta -\veps_n \}   )$ is smaller than
	\[   \mu \left(\left\{ z \in \R^d \,: \,  \exists \, \tilde z \text{ s.t. } |z-\tilde z | < \veps_n \text{ and } |\psi_r(z) - \psi_r(\tilde z)| \geq \frac{1}{3}\eta -\veps_n \right\}   \right).  \]
	Since the function $\psi_r$ is Lipschitz and $\veps_n \rightarrow 0$ as $n \rightarrow \infty$, it follows that
	\[ \lim_{n \rightarrow \infty}   \mu \left( \left\{ z \in \R^d \,: \, \exists \, \tilde z \text{ s.t. } |z-\tilde z | < \veps_n \text{ and } |\psi_r(z) - \psi_r(\tilde z)| \geq \frac{1}{3}\eta -\veps_n \right\}   \right)= 0. \] 
	From all the above inequalities it follows that
	\[ \limsup_{n \rightarrow \infty }  \mu \left( \left\{ z \in \R^d \: : \: |\mathrm{T}_n(z) -\mathrm{T}(z) | \geq \eta  \right\} \right)  \leq  2r.\]
	Since $r>0$ was arbitrary, we conclude that
	\[  \lim_{n \rightarrow \infty }  \mu \left( \left\{ z \in \R^d \: : \: |\mathrm{T}_n(z) -\mathrm{T}(z) | \geq \eta  \right\} \right) =0, \] 
	as we wanted to prove.

	Conversely, if $\mathrm{T}_n$ converges in $\mu$-measure, then we can assume without loss of generality that the convergence is actually $\mu$-a.e. (as we can work along subsequences). It follows now that for every $\phi \in C_b(\R^d \times \R^m)$, 
	\begin{align*}
	\lim_{n \rightarrow \infty} \int \phi(z, \theta) \, d\pi_n(z, \theta) &=\lim_{n \rightarrow \infty} \int \phi(z, \mathrm{T}_n(z)) \, d\mu(z)  
	\\&=  \int \lim_{n \rightarrow \infty} \phi(z, \mathrm{T}_n(z)) \, d\mu(z) 
	\\& = \int  \phi(z,  \lim_{n \rightarrow \infty} \mathrm{T}_n(z)) \, d\mu(z)
	\\& =\int  \phi(z,  \mathrm{T}(z)) \, d\mu(z)
	\\& = \int \phi(z, \theta) \, d\pi(z, \theta),
	\end{align*}
	where the second equality follows from the dominated convergence theorem, and the third equality follows from the continuity of $\phi$. This shows that $\pi_n$ converges weakly to $\pi$.
\end{proof}

\begin{remark}
	Lemma \ref{lem:CharacterizationTL^0} is analogous to the characterization of the $TL^p$ convergence in Proposition 3.12. in \cite{NGTSlepcev}. In Lemma \ref{lem:CharacterizationTL^0}, however, we have restricted our attention to the case where the base measure for the entire approximating sequence is the same (i.e., $\mu$).
\end{remark}

We recall the definition of convergence in the weak topology of the Hilbert space $\mathbb{L}^2(\R^d : \R^m; \mu)$.

\begin{defn}
	\label{def:WeakL2}
	We say that the sequence $\{ \mathrm{T}_n \}_{n \in \N}$ converges weakly in $\mathbb{L}^2(\R^d : \R^m; \mu)$ to $\mathrm{T}$ if
	\[ \lim_{n \rightarrow \infty} \int \mathrm{T} _n(z) \cdot g(z) \,d\mu(z) =  \int \mathrm{T}(z) \cdot g(z) \, d\mu(z)  \, \quad \forall g \,\in \mathbb{L}^2(\R^d : \R^m; \mu).   \]
\end{defn}

The next two lemmas are well-known results in measure theory and functional analysis.

\begin{lemma}
	\label{lem:MeasureWeakL2}
	Suppose that $\mathrm{T}_n \rightarrow \mathrm{T}$ in $\mu$-measure (as defined in Lemma \ref{lem:CharacterizationTL^0}) and that 
	\[ \sup_{n \in \N} \int |\mathrm{T}_n(z)|^2\,  d\mu(z) < \infty.   \]
	Then $\mathrm{T}_n$ converges weakly in $\mathbb{L}^2(\R^d : \R^m; \mu)$ to $\mathrm{T}$, as $n \rightarrow \infty$.
\end{lemma}

\begin{proof}
	From the second moment condition we deduce that the sequence $\{ \mathrm{T}_n \}_{n \in \N}$ is uniformly integrable. This, together with the dominated convergence theorem, allows us to conclude that
	\[ \lim_{n \rightarrow \infty} \int \mathrm{T}_n(z) \cdot g(z)  \,d\mu(z) =   \int   \mathrm{T}(z)  \cdot g(z)  \,d\mu(z)   \]
	for every $g \in \mathbb{L}^2(\R^d : \R^m; \mu) $, which is what we wanted to show.

\end{proof}

\begin{lemma}
	\label{lem:weak+norm=strong}
	If $\mathrm{T}_n$ converges weakly in $\mathbb{L}^2(\R^d : \R^m; \mu)$ to $\mathrm{T}$, and in addition we have
	\[ \lim_{n \rightarrow \infty} \int |\mathrm{T}_n(z)|^2 d\mu(z)= \int |\mathrm{T}(z)|^2 \,d\mu(z),  \]
	then 
	\[ \lim_{n \rightarrow \infty} \lVert \mathrm{T}_n - \mathrm{T}\rVert_{\mathbb{L}^2(\R^d : \R^m; \mu)}=0.  \]
\end{lemma}
\begin{proof}
	Expanding the square, we get
	\[ \int |\mathrm{T}_n(z) - \mathrm{T}(z)|^2 \, d\mu(z) =  \int |\mathrm{T}_n(z)|^2 \, d\mu(z) +  \int |\mathrm{T}(z)|^2 \, d\mu(z)  - 2 \int \mathrm{T}_n(z) \cdot \mathrm{T}(z)  \, d\mu(z). \]
	The result follows now from the above display, the assumed consistency of second moments, and the fact that $\mathrm{T}_n$ converges weakly to $\mathrm{T}$ (see Definition~\ref{def:WeakL2}). 
\end{proof}

\section{More general loss functions $\ell(\cdot,\cdot)$}\label{sec:General-cost}
If the squared error loss in problem \eqref{eq:Oracle-Bayes-risk} is substituted with an arbitrary loss function $\ell(\cdot,\cdot)$, the resulting problem
\begin{equation}
\inf_{\delta: \R^d \rightarrow \R^m } \E_{(Z, \Theta)\sim P_{Z,\Theta} }\left[\ell(\delta(Z), \Theta) \right]  \qquad \mbox{subject to }\;\; \delta(Z) \sim G^*
\label{eq:GeneralLoss}
\end{equation}
can still be written as a standard OT problem in Monge form:
\begin{equation}
\min_{\delta \: : \: \delta_{\sharp} \mu = G^* } \int_{\R^d} c_\ell(\delta(z), z) \, d \mu(z)
\label{eq:OTGeneral}
\end{equation}
for the cost function
\[ c_\ell(\theta , z ) := \E [\ell(\theta, \Theta ) \: | \: Z=z ], \qquad \mbox{for} \;\;\theta \in \R^m, \: z\in {\R^d}. \]
The existence of solutions for \eqref{eq:Oracle-Bayes-risk} then reduces to proving existence of an OT map for \eqref{eq:OTGeneral}. 

Investigating the existence of OT maps for specific transport problems is an important topic in the  theory of OT. A general strategy that can be followed for proving existence of optimal maps (also called Monge maps) is based on the analysis of the optimality conditions of solutions to the Kantorovich relaxation~\cite[Chapters 1-3]{Villani2003} of the original Monge problem; an important property of Kantorovich relaxations is that they can be shown to have solutions under very mild lower-semicontinuity assumptions on the transportation cost function (see e.g.,~\cite[Chapter 5]{Villani}). Notice that the Kantorovich relaxation of \eqref{eq:OTGeneral} takes the form
\[  \min_{\pi \in \Gamma(\mu, G^*)} \int c_\ell(\theta, z) \, d\pi(z,\theta). \]
Under appropriate assumptions on the cost function and marginals of a general OT problem, one can show that a solution to the Kantorovich relaxation must be supported on a graph of a function, and from this one can infer the existence of a solution to the original Monge problem. In principle, one could attempt to carry out this program for the OT problem \eqref{eq:OTGeneral}, but we notice that the dependence of the cost $c_\ell(\cdot,\cdot)$ on the loss function $\ell$ and on the model $P_{Z , \Theta}$ may, in general, be rather intractable. For this reason, in this paper we have focused on one tractable and important case, namely, the setting of the squared error loss, for which we can prove a variety of theoretical results and discuss a variety of algorithmic consequences.

% which in the case of \eqref{} is 
% \[  \min_{ \pi \in \Gamma(\mu, G^*) } \int_{\R^d} c_\ell(\theta, z) d \pi(z,\theta);  \]
% here $\Gamma(\mu, G^*)$ denotes the set of couplings or transport plans between $\mu$ and $G^*$, i.e., the space of joint probability measures with marginals $\mu$ and $G^*$, respectively. In the variable $\pi$, the Kantorovich relaxation is a linear program   
%  One then considers the optimality conditions that this coupling must satisfy, and checks that one of teh marginals and the cost function satisfy some required conditions to conclude that the optimal coupling must be supported in the graph of a map, and thus can be realized by a transport map. 

\section{An empirical Bayes approach to estimating the OT-based denoiser $\delta^*$}
\label{sec:EmpiricalBayes}
{Suppose that we observe $Z_1,\ldots, Z_n$ from model~\eqref{eq:Mix-Mdl} where the unobserved latent variables are $\Theta_1,\ldots, \Theta_n$ drawn i.i.d.~from $G^*$. We assume here that $G^*$ is unknown and belongs to a (sub)-family $\Ps$ of $\Ps(\Omega)$, the space of all probability measures on $\Omega \subset \R^m$. In the following we discuss an approach to estimate the OT-based denoiser $\delta^*$ based on the observed data $Z_1,\ldots, Z_n$. We plan to pursue a more thorough analysis of this framework in future work.
	
	Our approach can be broken down into three steps: (a) first we estimate the unknown prior $G^*$, say by $\hat G$, using the method of maximum likelihood, and (b) then use $\hat G$ as a plug-in estimator to solve an empirical version of the Kantorovich relaxation problem in~\eqref{eqn:Kantorovich}. This yields an optimal coupling (based on the data) which can (c) then be used to define an estimator of the OT-based denoiser $\delta^*$. 
	
	Let us describe each step in a bit more detail now. 
	
	(a) We apply the method of maximum likelihood (ML) to estimate $G^*$. Marginally, the observations $Z_i$'s are i.i.d.~with density $f_{G^*}$ (as defined in~\eqref{eq:Marg-Dens}). A ML estimator is any $\hat G\in \Ps$ which maximizes the marginal likelihood of the observations~$(Z_i)_{i=1}^n$, i.e., 
	\begin{align}\label{eq-NPMLE}
	\hat G &\in \argmax_{G\in \Ps} \frac{1}{n}\sum_{i=1}^n\log f_{G}(Z_i).
	\end{align}
	Note that when $\Ps = \Ps(\Omega)$, the space of all probability measures on $\Omega \subset \R^m$, this estimator is called the nonparametric MLE (NPMLE) of $G^*$ and has been studied in detail in the statistics literature; see~\cite{kiefer1956consistency, Lindsay-1983, Lindsay-1995, Jiang-Zhang-2009, Soloff-Et-Al-2021} and the references therein. In particular, in this case~\eqref{eq-NPMLE} is an infinite dimensional convex optimization problem for which several algorithms have been proposed; see e.g.,~\cite{Laird-1978, Bohning-1999, Lashkari08, Soloff-Et-Al-2021, Zhang-et-Al-2022}. Moreover, this approach can be applied even when $\Ps \subsetneq \Ps(\Omega)$ and/or $\Ps$ is finite dimensional. In the empirical Bayes literature this approach falls under the general framework of $G$-modelling as we directly estimate the unknown prior $G^*$ (\cite{Efron-19}).

	(b) In our second step, given an estimate~$\hat G$ of the prior~$G^*$, empirical Bayes imitates the optimal Bayesian analysis \cite{Efron-19}. If $G^*$ were known, the Bayes estimator of $\Theta_i$ (under the squared error loss) would be the posterior mean $\E_{G^*}[\Theta_i\mid Z_i]$ as defined in~\eqref{eq:Bayes-Est} (here by $\E_{G^*}[\ldots]$ we emphasize the dependence on $G^*$). 
	The NPMLE~\eqref{eq-NPMLE} yields a fully data-driven, empirical Bayes estimate of this posterior mean via
	\begin{align}\label{eq-posterior-mean}
	\hat{\bar \theta}(Z_i) \coloneqq \E_{\hat G}\left[\hat \Theta_i\mid Z_i\right], \;\; \mathrm{~where~} \;\;\hat \Theta_i\sim \hat G \;\;\mathrm{~and~} \;\; Z_i\mid \hat \Theta_i = \theta\sim p(\cdot \mid \theta).
	\end{align}
	Once we obtain an estimator ($\hat{\bar \theta}$ as above) of $\bar \theta(\cdot)$, we can solve an empirical version of~\eqref{eqn:Kantorovich} defined via 
	\begin{equation}
	\hat \pi \in \argmin_{ \pi \in \Gamma(\mu_n, \hat G)} \int |\hat{\overline{\theta}}(z) - \vartheta |^2 \, d \pi( \mu_n, \vartheta),
	\label{eq:OT_PushForward2}
	\end{equation} 
	where $\mu_n := \frac{1}{n} \sum_{i=1}^n \delta_{Z_i}$ is the empirical distribution of the $Z_i$'s. If $\hat G$ is the NPMLE over $\Ps(\Omega)$, the above computation is quite straightforward as it is known that $\hat G$ is finitely supported (see~\cite{Lindsay-1995}) and thus~\eqref{eq:OT_PushForward2} reduces to a discrete-discrete OT problem which can be solved using the various computational OT tools available in the literature (see e.g.,~\cite{Peyre-2019}).  
	
	(c) The optimal coupling $\hat \pi$ obtained in~\eqref{eq:OT_PushForward2} can now be used to construct an estimator of the OT-based denoiser $\delta^*(\cdot) \equiv \nabla \varphi^*(\hat {\bar \theta}(\cdot))$ (see~\eqref{eq:delta-OT} and~\eqref{eq:OT_PushForward}) via the {\it barycentric projection} of $\hat \pi$:
	\begin{equation}
	\delta_{\hat \pi}(z) := \int_{\Omega} \vartheta  \, d \hat \pi( \vartheta | z), \qquad \mbox{for}\;\; z \in {\mathcal{Z}}.
	\end{equation}
	We conjecture that $\delta_{\hat \pi}$ will be a consistent estimator of $\delta^*$; see~\cite{Deb-Et-Al-2021} and~\cite{Slawski-Sen-2022} where the barycentric projection estimator has been investigated and shown to be consistent for estimating OT maps.
	%\begin{equation}
	%    \hat \pi \in \argmin_{ T:  T_{\sharp }({\hat{\overline{\theta}}}_{\sharp} \mu_n)=\hat G} \int |\theta- T(\theta) |^2 \, d \hat{\overline{\theta}}_{\sharp} \mu_n (\theta),
	% \label{eq:OT_PushForward2}
	%\end{equation} 
	%between the measures $\overline{\theta}_{\sharp} \mu_n$ and $\hat G$, where $\mu_n := \frac{1}{n} \sum_{i=1}^n \delta_{Z_i}$ is the empirical distribution of the observed data $Z_1,\ldots, Z_n$. 
}

%USE THE BELOW OPTIONS IN CASE YOU NEED AUTHOR YEAR FORMAT.
% \bibliographystyle{abbrvnat}
% \bibliography{References}

\end{document}